\documentclass[11pt]{amsart}

\usepackage{amsmath,amsthm,amsfonts,amssymb, amscd}
\usepackage{amsthm}
\usepackage[all]{xy}
\usepackage{bbold}

\newtheorem{thm}{Theorem}[section]
\newtheorem{lemma}[thm]{Lemma}
\newtheorem{lem}[thm]{Lemma}
\newtheorem{prop}[thm]{Proposition}
\newtheorem{defn}[thm]{Definition}
\newtheorem{cor}[thm]{Corollary}
\newtheorem{remark}[thm]{Remark}

\author[Iv\'an Angiono]{Iv\'an Angiono$^{\star}$}
\thanks{$^{\star}$ Partially supported by CONICET, FONCyT-ANPCyT, Secyt (UNC)}
\address{FaMAF-CIEM (CONICET), Universidad Nacional de C\'ordoba, Medina Allende s/n, 5000 C\'ordoba, Argentina}
\email{angiono@famaf.unc.edu.ar}

\author[Mikhail Kochetov]{Mikhail Kochetov$^{\star\star}$}
\address{Department of Mathematics and Statistics,
Memorial University of Newfoundland, St. John's, NL, A1C5S7, Canada}
\email{mikhail@mun.ca}

\author[Mitja Mastnak]{Mitja Mastnak$^{\star\star}$}
\thanks{$^{\star\star}$Supported by the Natural Sciences and Engineering Research Council (NSERC) of Canada, Discovery Grants 341792-2013 and 371994-2014.}
\address{Department of Mathematics and Computing Science,
Saint Mary's University, 923 Robie Street, Halifax, NS, B3H3C3, Canada}
\email{mmastnak@cs.smu.ca}

\title{On rigidity of Nichols algebras}



\def\ep{\varepsilon}
\def\ot{\otimes}
\def\De{\Delta}
\def\na{\partial}

\def\gr{\operatorname{gr}}
\def\id{\operatorname{id}}
\def\Hom{\operatorname{Hom}}

\def\ad{\operatorname{ad}}
\def\coh{\operatorname{H}}
\def\Tot{\operatorname{Tot}}
\def\Ker{\operatorname{Ker}}

\def\m{\operatorname{m}}
\def\Homv{\operatorname{Hom}}

\def\B{\textbf{B}}
\def\F{\textbf{F}}
\def\bG{\textbf{G}}

\newcommand{\V}{\mathcal{V}}
\newcommand{\M}{\mathcal{M}} 
\newcommand{\yd}{\mathcal{YD}}

\newcommand\ord{\operatorname{ord}}
\newcommand{\vi}{\textbf{(i)} }
\newcommand{\vii}{\textbf{(ii)} }
\newcommand{\viii}{\textbf{(iii)} }
\newcommand{\viv}{\textbf{(iv)} }
\newcommand{\vv}{\textbf{(v)} }

\newcommand{\ydg}{^{\kk\Gamma}_{\kk\Gamma}\mathcal{YD}}

\newcommand\G{\mathbb{G}}

\newcommand\Z{\mathbb{Z}}
\newcommand\N{\mathbb{N}}
\def\cI{\mathcal{I}}

\def\cB{\mathcal{B}}
\def\cU{\mathcal{U}}
\def\cO{\mathcal{O}}

\def\cR{\mathcal{R}}

\def\cD{\mathcal{D}}

\newcommand{\kk}{\Bbbk}

\newcommand\supp{\operatorname{supp}}
\def\ot{\otimes}

\begin{document}

\begin{abstract}
We study deformations of graded braided bialgebras using cohomological methods.
In particular, we show that many examples of Nichols algebras, including the finite-dimensional ones
arising in the Andruskiewitsch-Schneider program of classification of pointed Hopf algebras,
are rigid. This result can be regarded as nonexistence of ``braided Lie algebras'' with nontrivial bracket.
\end{abstract}

\maketitle

\section{Introduction}

Let $\kk$ be a field of characteristic $0$ and $V$ a $\kk$-vector space. The symmetric algebra $S(V)=\bigoplus_{n\ge 0}S^n(V)$ is a graded bialgebra
by declaring the elements of $V$ {\em primitive}, i.e. $\Delta(x)=x\otimes 1+1\otimes x$ for all $x\in V$, and
extending to a morphism of (unital) algebras $\Delta\colon S(V)\to S(V)\ot S(V)$.
Then Lie brackets on $V$ are in one-to-one correspondence with graded deformations of $S(V)$ as a bialgebra (or just as an augmented algebra).

We are interested in graded deformations of bialgebras generalizing $S(V)$, namely, the Nichols algebras of braided vector spaces, which have become
prominent in the theory of Hopf algebras (see the survey \cite{AS-survey} and references therein). Recall that a {\em braided vector space}
is a vector space $V$ equipped with a linear isomorphism $c\colon V\ot V\to V\ot V$ that satisfies the {\em braid equation}
\[
(c\ot\id)(\id\ot c)(c\ot\id)=(\id\ot c)(c\ot\id)(\id\ot c),
\]
where $\id=\id_V$. The {\em Nichols algebra} of $(V,c)$, denoted by $\cB(V,c)$ or just $\cB(V)$ if the braiding is clear from the context,
is the unique (up to isomorphism) graded braided bialgebra
$\cB=\bigoplus_{n\ge 0} \cB_n$ with $\cB_0=\kk$, $\cB_1=V$ such that the restriction of the braiding of $\cB$ to $V$ is $c$, $\cB$ is generated by $V$ as an algebra, and $V$ coincides with the space $P(\cB)$ of primitive elements of $\cB$.

In the case of {\em symmetric} braiding, i.e., $c^2=\id$, the concept of braided Lie algebra is well understood \cite{Gu,BFM1,Kharch,Ko,KS_Lie}.
This includes the usual Lie algebras (when $c$ is the flip $v\ot w\mapsto w\ot v$), Lie superalgebras
(when $V$ is graded by $\Z_2$ and $c$ is the signed flip $v\ot w\mapsto (-1)^{p(v)p(w)}w\ot v$ where $p$ denotes parity) and color Lie superalgebras.
It follows from Kharchenko's version of PBW Theorem \cite[Theorem 7.1]{Kharch} that such Lie structures on $(V,c)$ are in one-to-one correspondence with graded deformations of $\cB(V,c)$
as a braided bialgebra with a fixed braiding (see Section \ref{s:symmetric}).

It is an important and difficult question for what finite-dimensional braided vector spaces the Nichols algebra is also finite-dimensional.
This condition puts severe restrictions on $c$. For example, in the case of signed flip, this happens if and only if the even part of $V$ is zero,
in which case the Nichols algebra is the exterior algebra $\Lambda(V)$ and there are no nontrivial graded deformations.

We believe that such rigidity is typical for finite-dimensional Nichols algebras.
We establish it for a wide class of symmetric braidings (Theorem \ref{thm:symmetric})
using the description of finite-dimensional triangular Hopf algebras by Etingof and Gelaki \cite{EG1,Ge,EG2}.
We also establish a sufficient condition of rigidity (Theorem \ref{thm:suff_cond}) using cohomological techniques,
and verify that it is satisfied for finite-dimensional Nichols algebras in the Yetter-Drinfeld category $\ydg$
over an abelian group $\Gamma$ (Theorem \ref{thm:homzero}) using a description of these Nichols algebras in terms of generators and relations \cite{Ang-crelle}.
It follows that any finite-dimensional Nichols algebra arising from a diagonal braiding, i.e., a braiding of the form $c(x_i\ot x_j)=q_{ij}x_j\ot x_i$
where $\{x_1,\ldots,x_\theta\}$ is a basis of $V$ and $q_{ij}\in\kk^\times$, does not admit nontrivial graded deformations (Theorem \ref{thm:diagonal}).

It should be mentioned that the so-called {\em bosonizations} of these Nichols algebras often admit nontrivial graded deformations (or ``liftings''),
as has been shown by Andruskiewitsch and Schneider in the course of their program of classification of pointed Hopf algebras \cite{AS1}.

Our sufficient condition also applies to some interesting infinite-dimensional Nichols algebras (see Section \ref{s:other}) and other braided bialgebras close to Nichols algebras (Theorem \ref{thm:diagonal-pre-Nichols}).
This may explain the difficulty of constructing new examples in \cite{Ar}, where an attempt is made to define and study braided Lie algebras for non-symmetric braiding.

\section{Preliminaries}

\subsection{Braided tensor categories}

It is often more convenient to work in a category rather than with a stand-alone braided vector space.
By a {\em tensor category} we always mean a strict monoidal $\kk$-linear category, see e.g. \cite{McL} for details.
We are mostly interested in categories of $\kk$-vector spaces endowed with some additional structure.
To simplify notation, we omit associativity isomorphisms and parentheses in tensor products.
In particular, we denote the tensor powers of an object $V$ by $V^{\ot n}$ for all $n\geq 0$, where $V^{\ot 0}$ is the unit object.

A {\em braided tensor category} is a tensor category $\V$ with a {\em braiding},
i.e. a natural family of isomorphisms $c_{V,W}\colon V\ot W\to W\ot V$ in $\V$ satisfying the so-called hexagon axioms:
\[
\begin{array}{lcl}
c_{U,V\ot W}=(\id_V\ot c_{U,W})(c_{U,V}\ot \id_W)&\text{and}& 
c_{U\ot V,W}=(c_{U,W}\ot \id_V)(\id_U\ot c_{V,W}),
\end{array}
\]
for all $U,V,W$ in $\V$. The braid equation follows: 
\[
(c_{V,W}\ot\id_U)(\id_V\ot c_{U,W})(c_{U,V}\ot \id_W)=(\id_W\ot c_{U,V})(c_{U,W}\ot \id_V)(\id_U\ot c_{V,W}).
\] 
The category is said to be {\em symmetric} if  $c_{W,V} c_{V,W}=\id_{V\ot W}$ for all $V$, $W$ in $\V$.

The most well known braided tensor categories are the category of (co)modules over a (co)quasitriangular bialgebra and
the category of Yetter-Drinfeld modules over a Hopf algebra with bijective antipode. We will now briefly recall the relevant definitions and fix notation;
details can be found in textbooks such as \cite{M,KS}. We use the standard Sweedler notation for coalgebras and comodules.

A {\em coquasitriangular (CQT) bialgebra} is a pair $(H,\beta)$ where $H$ is a bialgebra and $\beta$
is a bilinear form $H\times H\to\kk$ that is invertible with respect to convolution and satisfies
\begin{align*}
\beta(h_{(1)},k_{(1)})h_{(2)}k_{(2)}&=\beta(h_{(2)},k_{(2)})k_{(1)}h_{(1)},\\
\beta(hk,\ell)&=\beta(h,\ell_{(1)})\beta(k,\ell_{(2)}),\\
\beta(\ell,hk)&=\beta(\ell_{(2)},h)\beta(\ell_{(1)},k),
\end{align*}
for all $h,k,\ell\in H$. The category of right comodules $\M^H$ is braided as follows:
\begin{align}\label{eq:braiding_CQT}
c_{V,W}(v\ot w) &= \beta(v_{(1)},w_{(1)})w_{(0)}\ot v_{(0)}, & \mbox{for all }v\in V, & \, w\in W.
\end{align}
Similarly, the category of left comodules ${}^H\M$ is braided by
\begin{align*}
c_{V,W}(v\ot w) &= \beta(w_{(-1)},v_{(-1)})w_{(0)}\ot v_{(0)}, & \mbox{for all }v\in V,& \, w\in W.
\end{align*}
If $G$ is a group then the Hopf algebra $H=\kk G$ admits a CQT structure $\beta$ if and only if $G$ is abelian.
In this case the possible maps $\beta$ are just linear extensions of bicharacters $G\times G\to\kk^\times$.
Right $H$-comodules are just $G$-graded vector spaces,
$V=\bigoplus_{g\in G} V_g$, and the braiding is given by $v\ot w\mapsto\beta(g,h) w\ot v$ for all $v\in V_g$, $w\in W_h$, $g,h\in G$.

An object $V$ of the Yetter-Drinfeld category ${}^H_H\yd$ is simultaneously a left module and a left comodule such that the following compatibility condition holds:
\begin{align*}
h_{(1)}v_{(-1)}\ot h_{(2)}\cdot v_{(0)}&=(h_{(1)}\cdot v)_{(-1)}h_{(2)}\ot(h_{(1)}\cdot v)_{(0)} &
\mbox{for all }&v\in V, \, h\in H.
\end{align*}
A morphism is a linear map preserving both action and coaction. The braiding is given by
\[
c_{V,W}\colon v\ot w\mapsto  v_{(-1)}\cdot w\ot v_{(0)}.
\]
The category of right Yetter-Drinfeld modules $\yd^H_H$ is defined in a similar manner.
If $\Gamma$ is a group and $H=\kk\Gamma$ then an object in ${}^H_H\yd$ is just a $\Gamma$-graded vector space with a left action of $\Gamma$
such that $g\cdot V_h=V_{ghg^{-1}}$, for all $g,h\in\Gamma$. The braiding is given by $v\ot w\mapsto g\cdot w\ot v$, for all $v\in V_g$, $w\in W$.
In particular, if $\Gamma$ is abelian then the semisimple objects in $\yd^H_H$ are vector spaces graded by the direct product $\Gamma\times\widehat{\Gamma}$
where $\widehat{\Gamma}$ is the character group of $\Gamma$.
For a vector space $V$ with such a grading, we will denote the homogeneous component of degree $(g,\chi)$ by $V^\chi_g$.
The braiding becomes $v\ot w\mapsto\psi(g) w\ot v$, for all $v\in V^\chi_g$ and $w\in W^\psi_h$.

If a CQT bialgebra $(H,\beta)$ is a Hopf algebra then its antipode is bijective.
Moreover $\M^H$ can be regarded as a full subcategory of the Yetter-Drinfeld category $\yd^H_H$ if we define the right action of $H$
on a right comodule $V$ by means of the usual left action of $H^*$ and the homomorphism of algebras $H^\mathrm{op}\to H^*\colon h\mapsto\beta(\cdot,h)$, i.e.,
$v\cdot h=\sum\beta(v_{(1)},h)v_{(0)}$, for all $v\in V$, $h\in H$. Similarly, ${}^H\M$ can be regarded as a full subcategory of ${}^H_H\yd$.

If $(U,c)$ is a finite-dimensional braided vector space then the FRT construction \cite{KS,Tk} yields a CQT bialgebra $(H,\beta)$
such that $U\in\M^H$ and $c=c_{U,U}$ where $c_{U,U}$ is given by \eqref{eq:braiding_CQT}. Moreover, for any $V,W\in\M^H$ and a linear map $f\colon V\to W$ that
{\em commutes with the braiding with $U$} in the sense that $(f\ot\id)c_{U,V}=c_{U,W}(\id\ot f)$ and $(\id\ot f)c_{V,U}=c_{W,U}(f\ot\id)$,
there exists a biideal $I$ of $H$ contained in the left and right kernels of the bilinear form $\beta$
such that $f$ is a morphism in $\M^{H/I}$ \cite[Corollary 1.9]{Tk}. Hence, replacing $(H,\beta)$ by $(\bar{H},\bar{\beta})$,
where $\bar{H}$ is the quotient of $H$ by the largest biideal contained in the left and right kernels of $\beta$ and where $\bar{\beta}$ is induced by $\beta$,
we obtain a braided category, $\M^{\bar{H}}$, that contains $(U,c)$ and all linear maps that commute with the braiding with $U$.

There is a Hopf algebra version of the above construction --- see e.g. \cite{Tk} and references therein --- for braided vector spaces
satisfying a certain condition, called {\em rigidity} in \cite{Tk}, which allows us to define the braiding operators $c_{U,U^*}$, $c_{U^*,U}$ and $c_{U^*,U^*}$,
where $U^*$ is the dual space. Namely, there exists a CQT Hopf algebra $(H,\beta)$ such that $U\in\M^H$ and $c=c_{U,U}$.
Again, any linear map that commutes with the braiding with $U$ can be included in the category $\M^{H/I}$ where $I$ is a Hopf ideal contained
in the left and right kernels of $\beta$, see the proof of \cite[Proposition 5.4]{Tk}. Since the largest biideal contained in the kernels of
$\beta$ is automatically a Hopf ideal, we obtain a CQT Hopf algebra $\bar{H}$ such that $\M^{\bar{H}}$ includes $(U,c)$ and
all linear maps that commute with the braiding with $U$.

We are especially interested in the case of diagonal braiding: $c(x_i\ot x_j)=q_{ij}x_j\ot x_i$
where $\{x_1,\ldots,x_\theta\}$ is a basis of $U$ and $q_{ij}\in\kk^\times$. Here we can take $H=\kk G$, where $G$ is the free abelian group $\Z^\theta$, and
define the bicharacter $\beta$ by setting $\beta(e_i,e_j)=q_{ij}$, where $\{e_1,\ldots,e_\theta\}$ is the standard basis of $\Z^\theta$.
If we make $U$ a $G$-graded vector space by declaring $x_i\in U_{e_i}$ then we get $c=c_{U,U}$ in $\M^H$. Alternatively, we can make $U$ an object
of $\ydg$ for each abelian group $\Gamma$ containing elements $g_1,\ldots,g_\theta$ such that there exist characters
$\chi_1,\ldots,\chi_\theta\in\widehat{\Gamma}$ satisfying $\chi_j(g_i)=q_{ij}$; then we declare $x_i\in U^{\chi_i}_{g_i}$ and get $c=c_{U,U}$ in $\ydg$.
We can choose the group $\Gamma$ so that it is generated by $g_1,\ldots,g_\theta$ and the characters $\chi_1,\ldots,\chi_\theta$ separate points of $\Gamma$.
It is easy to see that in this case a linear map $f\colon V\to W$ commutes with the braiding with $U$ if and only if $f$ is a morphism in $\ydg$.

\subsection{Braided bialgebras}\label{ss:braided_bialg}

A {\em bialgebra} in a braided tensor category $\V$ with unit object $\mathbb{1}$ is an object $\cB$ with four morphisms: multiplication $m\colon \cB\ot \cB\to \cB$,
unit $u\colon \mathbb{1}\to \cB$, comultiplication $\Delta\colon \cB\to \cB\ot \cB$ and counit $\ep\colon \cB\to \mathbb{1}$ such that
$(\cB,m,u)$ is a unital algebra, $(\cB,\Delta,\ep)$ is a counital coalgebra, and the following compatibility conditions hold:
\begin{align*}
\Delta m&=(m\ot m)(\id_\cB\ot c_{\cB,\cB}\ot\id_\cB)(\Delta\ot\Delta), &  \ep u&=\id_{\mathbb{1}}, & \ep m&=\ep\ot\ep, & \Delta u&=u\ot u.
\end{align*}
Note that the braiding appears only in the compatibility condition involving $m$ and $\Delta$.

One can define a {\em braided bialgebra} without reference to any categories \cite{Tk}: it is a braided vector space $(\cB,c)$ with four linear maps,
$m\colon \cB\ot \cB\to \cB$, $u\colon\kk\to \cB$, $\Delta\colon \cB\to \cB\ot \cB$ and $\ep\colon \cB\to\kk$, that commute with 
the braiding induced by $c$ among the tensor powers of $\cB$ and satisfy the following conditions:
$(\cB,m,u)$ is a unital algebra, $(\cB,\Delta,\ep)$ is a counital coalgebra,
$u$ is a counital coalgebra map, $\ep$ is a unital algebra map, and finally
$\Delta m=(m\ot m)(\id_\cB\ot c\ot\id_\cB)(\Delta\ot\Delta)$.

Obviously, a bialgebra $\cB$ in a braided tensor category consisting of vector spaces and linear maps (such as $\M^H$ or ${}^H_H\yd$) satisfies
the definition of braided bialgebra with $c=c_{\cB,\cB}$. Conversely, it is shown in \cite{Tk} that any finite-dimensional braided bialgebra
$(\cB,m,u,\Delta,\ep,c)$ can be included in the category $\M^H$ over a suitable CQT bialgebra (Hopf algebra if $c$ is rigid) $H$ such that
$m,u,\Delta,\ep$ are morphisms in $\M^H$ and $c=c_{\cB,\cB}$ in $\M^H$.

We are mainly interested in the case of the Nichols algebra $\cB(V)$ of a finite-dimensional vector space $V$ with a rigid braiding $c$,
which is a braided Hopf algebra, not necessarily finite-dimensional but equipped with a grading over non-negative integers whose components are finite-dimensional. 
It can be constructed as the quotient of the tensor algebra $T(V)$ by a graded biideal $\cI(V)$ \cite[Proposition 2.2]{AS-survey}, which is determined by the braiding $c$;
indeed the homogeneous components of $\cI(V)$ are the kernels of the so-called {\em quantum symmetrizers} on the tensor powers of $V$
\cite[Proposition 2.11]{AS-survey}. This construction can be carried out either with the stand-alone braided vector space $(V,c)$ or in a suitable
braided category of comodules or Yetter-Drinfeld modules.

\subsection{Graded deformations and liftings}

We review the theory of formal graded deformations and liftings from \cite{MW}, but in a slightly more general setting.
The theory of formal bialgebra deformations was introduced by Gerstenhaber and Schack \cite{GS},
while the graded version and its connection to liftings was considered by Du, Chen and Ye \cite{DCY}. In this context, a {\em graded bialgebra} 
will mean a bialgebra $\cB$ in a braided tensor category $\V$ (consisting of vector spaces and linear maps) equipped with a grading, as an object in $\V$, over non-negative integers, $\cB=\bigoplus_{n\ge 0}\cB_n$, which is at the same time an algebra and a coalgebra grading, i.e., $\cB_i\cB_j\subseteq\cB_{i+j}$ and $\Delta(\cB_k)\subseteq\bigoplus_{i+j=k}\cB_i\ot\cB_j$, for all $i,j,k\ge 0$.

Let $t$ be an indeterminate and consider the polynomial algebra $\kk[t]$ equipped with its standard grading, i.e., $t$ has degree $1$. 
By extending scalars from $\kk$ to $\kk[t]$, the braided tensor category $\V$ gives rise to the braided tensor category $\V_{\kk[t]}$.
A (formal) {\em graded deformation} of a graded bialgebra $(\cB,m,\Delta)$ in $\V$ is a $\kk[t]$-linear graded structure
$(m_t, \Delta_t)$ on $\cB[t]=\cB\ot\kk[t]$ such that $(\cB[t],m_t,\Delta_t)$ is a graded bialgebra in $\V_{\kk[t]}$. 

We say that two graded deformations, $(\cB[t], m_t, \Delta_t)$ and $(\cB[t], m'_t, \Delta'_t)$, are \emph{equivalent}
if there exists a $\kk[t]$-linear graded bialgebra isomorphism $f\colon (\cB[t], m_t, \Delta_t)\to (\cB[t], m'_t, \Delta'_t)$. 

A {\em lifting} $(\cU,\pi)$ of $\cB$ consists of a filtered bialgebra $\cU$ and a filtered vector space isomorphism $\pi\colon \cU\to \cB$ such that
$\gr\pi\colon \gr \cU\to \gr \cB = \cB$ is an isomorphism of graded bialgebras.  An \emph{equivalence} between liftings $(\cU,\pi)$ and $(\cU',\pi')$
is a filtered bialgebra isomorphism $f\colon \cU\to \cU'$ such that $\gr\pi \circ \gr f = \gr \pi'$.

A graded deformation is given by a sequence of pairs of maps
$(m_i, \Delta_i)$, $i\geq 0$, of degree $-i$ such that
$m_t|_{\cB\otimes \cB}=m+\sum_{i\ge 1} m_it^i$ and $\Delta_t|_{\cB}=
\Delta+\sum_{i\ge 1} \Delta_i t^i$.  We also denote $(m_0,\Delta_0)=
(m,\Delta)$.  A graded deformation $(\cB[t], m_t, \Delta_t)$ defines a lifting
$(\cU,\pi)$, where $\cU$ is $\cB$ as a filtered vector space, $\pi$ is identity, and
$(m_\cU,\Delta_\cU)=(m_t,\Delta_t)|_{t=1}$.

If $(\cU,\pi)$ is a lifting, then the linear maps
$\tilde{m}\colon \cB\otimes \cB\stackrel{\pi^{-1}\otimes \pi^{-1}}{\longrightarrow} \cU\ot \cU \stackrel{m_\cU}{\to} \cU\stackrel{\pi}{\to} \cB$
and $\tilde{\Delta}\colon \cB \stackrel{\pi^{-1}}{\to} \cU \stackrel{\Delta_\cU}{\to} \cU\ot \cU\stackrel{\pi\ot\pi}{\longrightarrow} \cB\ot \cB$
decompose into direct sums of homogeneous components $m_i, \Delta_i$ of degrees $-i$ for $i\geq0$,
and the structure maps $(m_t, \Delta_t)=(\sum_i m_i t^i, \sum_i \Delta_i t^i)$ on $\cB[t]$ define a formal graded deformation of $\cB$.

Up to equivalence, these correspondences are inverses of each other.

\subsection{Graded bialgebra cohomology}

Let $\cB$ be a bialgebra in $\V$. Consider the bisimplicial complex $\B=(\B^{p,q})_{p,q\ge 0}$,
\[
\B^{p,q}=\Homv(\cB^{\ot p}, \cB^{\ot q}).
\]
The left and right diagonal actions and coactions of $\cB$ on $\cB^{\ot n}$ will be denoted by $\lambda_l, \lambda_r, \rho_l,\rho_r$, respectively.
Note that they involve the braiding. The horizontal faces
\[
\na^h_i\colon \Homv(\cB^{\ot p}, \cB^{\ot q})\to \Homv(\cB^{\ot (p+1)}, \cB^{\ot q})
\]
and degeneracies
\[
\sigma^h_i\colon \Homv(\cB^{\ot (p+1)}, \cB^{\ot q})\to \Homv(\cB^{\ot p}, \cB^{\ot q})
\]
are those for computing Hochschild cohomology:
\begin{eqnarray*}
\na^h_{0} f&=&\lambda_l(\id\ot f),\\
\na^h_{i} f&=& f(\id\ot\ldots\ot m\ot\ldots \ot \id),\; 1\le i\le p,\\
\na^h_{p+1}f&=& \lambda_r(f\ot \id),\\
\sigma^h_{i} f&=& f(\id\ot\ldots\ot u\ot\ldots\ot \id);
\end{eqnarray*}
the vertical faces
\[
\na^c_j\colon \Homv(\cB^{\ot p}, \cB^{\ot q})\to \Homv(\cB^{\ot p}, \cB^{\ot (q+1)})
\]
and degeneracies
\[
\sigma^c_j\colon \Homv(\cB^{\ot p}, \cB^{\ot (q+1)})\to \Homv(\cB^{\ot p}, \cB^{q})
\]
are those for computing coalgebra (Cartier) cohomology:
\begin{eqnarray*}
\na^c_{0} f&=&(\id\ot f)\rho_l,\\
\na^c_{j} f&=& (\id\ot\ldots\ot \De\ot\ldots \ot \id)f,\; \ 1\le j\le q,\\
\na^c_{q+1}f&=& (f\ot \id)\rho_r,\\
\sigma^c_{i} f&=& (\id\ot\ldots\ot\ep\ot\ldots\ot \id)f.
\end{eqnarray*}
The vertical and horizontal differentials are given by the usual alternating sums
\[
\na^h=\sum (-1)^i\na^h_i, \qquad \na^c=\sum (-1)^j\na^c_j.
\]
By abuse of notation we identify a cosimplicial bicomplex with its associated cochain bicomplex.
The {\em bialgebra cohomology} of $\cB$ is then defined as
\[
\coh_{\mathrm{b}}^*(\cB)=\coh^*(\Tot\B).
\]
where
\[
\Tot\B = \B^{0,0}\to \B^{1,0}\oplus\B^{0,1}\to\ldots \to\bigoplus_{p+q=n} \B^{p,q}\stackrel{\partial^b}{\to}\ldots
\]
and $\partial^b$ is given by the sign trick
(i.e., $\partial^b|_{\B^{p,q}}=\partial^h\oplus (-1)^p\partial^c\colon\B^{p,q}\to \B^{p+1,q}\oplus\B^{p,q+1}$).

Let $\B_0$ denote the bicomplex obtained from $\B$ by replacing the edges
by zeroes, i.e., $\B_0^{p,0}=0=\B_0^{0,q}$ for all $p,q$.
The {\em truncated bialgebra cohomology} is
\[
\widehat{\coh}_{\mathrm{b}}^*(\cB)=\coh^{*+1}(\Tot\B_0).
\]
For computations, it is convenient to use the normalized bicomplex $\B^+$, which is obtained from the cochain bicomplex $\B$
by replacing $\B^{p,q}=\Homv(\cB^{\ot p},\cB^{\ot q})$ with the intersection of degeneracies
\[
(\B^+)^{p,q}=(\cap \Ker\sigma^h_i)\cap (\cap \Ker\sigma^c_j)\simeq\Homv((\cB^+)^{\ot p},(\cB^+)^{\ot q}),
\]
where $\cB^+=\ker(\varepsilon)$. This change does not affect the cohomology.

We can describe the first two cohomology groups as follows:
\[
\widehat{\coh}_{\mathrm{b}}^1(\cB)=\{f\colon \cB^+\to \cB^+\;|\;f(ab)=af(b)+f(a)b, \ \De f(a)=a_{(1)}\ot f(a_{(2)})+f(a_{(1)})\ot a_{(2)}\}
\]
and
\[
\widehat{\coh}_{\mathrm{b}}^2(\cB)=\widehat{{\rm Z}}_{\mathrm{b}}^2(\cB)/\widehat{{\rm B}}_{\mathrm{b}}^2(\cB),
\]
where
\begin{align}
\widehat{{\rm Z}}_{\mathrm{b}}^2(\cB)= \big\{(f,g)\;\big|\;& f\colon \cB^+\ot \cB^+\to \cB^+,\  g\colon \cB^+\to \cB^+\ot \cB^+,\nonumber\\
&af(b,c)+f(a,bc)=f(ab,c)+f(a,b)c,\label{eq-assoc} \\
&c_{(1)}\ot g(c_{(2)})+(\id\ot\De)g(c)=(\De\ot\id)g(c)+g(c_{(1)})\ot c_{(2)},\label{eq-coassoc} \\
&(f\ot m)\De(a\ot b) -\De f(a,b)+(m\ot f)\De(a\ot b) = \label{eq-compat}\\
&\phantom{(f\ot m)\De(a\ot b)} -(\De a)g(b)+g(ab)-g(a)(\De b) \nonumber \big\}
\end{align}
and
\begin{align*}
\widehat{{\rm B}}_{\mathrm{b}}^2(\cB)= \big\{(f,g)\;\big|\; \exists h\colon \cB^+\to \cB^+, \,\, &f(a,b)=ah(b)-h(ab)+h(a)b, \\
   &g(c)=-c_{(1)}\ot h(c_{(2)})+\De h(c) -h(c_{(1)})\ot c_{(2)} \big\},
\end{align*}
where the elements $a,b,c$ range over $\cB^+$. All maps above are assumed to be morphisms in $\V$.
By $\De(a\ot b)$ we mean the braided coproduct in $\cB\ot\cB$, namely, $(\id\ot c_{\cB,\cB}\ot\id)(a_{(1)}\ot a_{(2)}\ot b_{(1)}\ot b_{(2)})$,
and we write $f(-,-)$ instead of $f(-\ot -)$.
In the resulting deformation (see the next subsection), Equation \eqref{eq-assoc} will correspond to associativity,
Equation \eqref{eq-coassoc} to coassociativity and Equation \eqref{eq-compat} to compatibility.

Now assume that $\cB$ is $\Z$-graded and let $\B_{\ell}$ denote the subcomplex of $\B$ consisting of homogeneous maps of degree $\ell$, i.e.,
\[
\B^{p,q}_{\ell}=\Homv(\cB^{\ot p},\cB^{\ot q})_\ell=\{f\colon \cB^{\ot p}\to \cB^{\ot q}\;|\;f\mbox{ is
homogeneous of degree } \ell\}.
\]
Complexes $(\B_0)_{\ell}$, $\B^+_{\ell}$ and $(\B^+_0)_{\ell}$ are defined analogously.
The graded bialgebra and truncated graded bialgebra cohomologies are then defined by:
\begin{align*}
&\coh^*_{\mathrm{b}}(\cB)_\ell=\coh^*(\Tot\B_{\ell})=\coh^*(\Tot\B^+_{\ell}),\\
&\widehat{\coh}_{\mathrm{b}}^*(\cB)_\ell=\coh^{*+1}(\Tot(\B_0)_{\ell})=\coh^{*+1}(\Tot(\B^+_0)_{\ell}).
\end{align*}
Note that if the support of the grading is finite, in particular if $\cB$ is finite-dimensional, then
\[
\coh_{\mathrm{b}}^*(\cB)=\bigoplus_{\ell\in\Z} \coh_{\mathrm{b}}^*(\cB)_\ell \ \mbox{ and } \
\widehat{\coh}_{\mathrm{b}}^*(\cB)=\bigoplus_{\ell\in\Z} \widehat{\coh}_{\mathrm{b}}^*(\cB)_\ell.
\]

\subsection{Cohomological aspects of graded deformations}

Given a graded deformation of $\cB$, let $r$ be the smallest positive
integer for which $(m_r, \Delta_r)\neq (0,0)$ (if such an $r$ exists).
Then $(m_r,\Delta_r)$ is a $2$-cocycle in $\widehat{{\rm Z}}^2_{\mathrm{b}}(\cB)_{-r}$.
Every nontrivial deformation is equivalent to one for which the
corresponding $(m_r, \Delta_r)$ represents a nontrivial cohomology class \cite{GS, DCY}.
Hence, if $\widehat{\coh}_{\mathrm{b}}^2(\cB)_{(\ell)}=0$ for all $\ell<0$, then $\cB$ is {\em rigid},
i.e., has no nontrivial graded deformations.

Conversely, given a positive integer $r$ and a $2$-cocycle
$(m',\Delta')$ in $\widehat{{\rm Z}}^2_{\mathrm{b}}(\cB)_{-r}$, the maps
$m+t^r m'$ and $\Delta+t^r\Delta'$ define a bialgebra structure on
$\cB[t]/(t^{r+1})$ over $\kk[t]/(t^{r+1})$.
There may or may not exist $(m_{r+k}, \De_{r+k})$, $k\geq 1$,
for which $m_t=m+t^rm'+\sum_{k\geq 1}t^{r+k}m_{r+k}$ and
$\De_t=\De+t^r\De'+\sum_{k\geq 1}t^{r+k}\De_{r+k}$
make $\cB[t]$ into a bialgebra over $\kk[t]$.

An {\em $r$-deformation} of $\cB$ is a graded deformation
of $\cB$ over $\kk[t]/(t^{r+1})$, i.e. a pair $(m^r_t,\Delta^r_t)$ defining a bialgebra structure on
$\cB[t]/(t^{r+1})$ over $\kk[t]/(t^{r+1})$ such that $(m^r_t,\Delta^r_t)|_{t=0}=(m,\Delta)$.
For any $2$-cocycle $(m',\Delta ')$ in $\widehat{{\rm Z}}^2_{\mathrm{b}}(\cB)_{-r}$,
there exists an $r$-deformation, given by $(m+t^rm', \Delta + t^r \Delta')$.

If a given $(r-1)$-deformation can be extended to an $r$-deformation, then all
ways of doing so are parametrized by $\widehat{\coh}_{\mathrm{b}}^2(\cB)_{-r}$.  More precisely, suppose that
$(\cB[t]/(t^r), \ m_t^{r-1}, \ \Delta_t^{r-1})$ is an
$(r-1)$-deformation, where
\begin{align*}
m_t^{r-1}&=m+t m_1+\ldots+t^{r-1} m_{r-1}, &  \Delta_t^{r-1}&=\Delta+t\Delta_1+\ldots +t^{r-1}\Delta_{r-1}.
\end{align*}
If
\[
D=(\cB[t]/(t^{r+1}), m_t^{r-1}+t^{r} m_{r},\Delta_t^{r-1}+t^{r}\Delta_{r})
\]
is an $r$-deformation, then
\[
D'=(\cB[t]/(t^{r+1}), m_t^{r-1}+t^{r}m'_{r},\Delta_t^{r-1}+t^{r}\Delta'_{r})
\]
is an $r$-deformation if and only if
$(m'_{r}-m_{r},\Delta'_{r}-\Delta_{r})\in\widehat{{\rm Z}}_{\mathrm{b}}^2(\cB)_{-r}.$
Note also that if $(m'_{r}-m_{r},\Delta'_{r}-\Delta_{r})\in\widehat{{\rm B}}_{\mathrm{b}}^2(\cB)_{-r}$, then deformations $D$ and $D'$ are
equivalent.

The obstruction to extend $r$-deformations to $(r+1)$-deformations lies in $\widehat{\coh}_{\mathrm{b}}^3(\cB)_{-r}$.

\section{The case of symmetric braiding}\label{s:symmetric}

Let $(V,c)$ be a braided vector space with $c^2=\id$. Then $\cB(V)$ is a quadratic algebra: it is the quotient of $T(V)$ by the ideal generated by
the elements $x\ot y-c(x\ot y)$, for $x,y\in V$.
If $c$ is the flip (respectively, signed flip) then $\cB(V)=S(V)$ (respectively, $S(V_0)\ot\Lambda(V_1)$) and the graded deformations of $\cB(V)$
are in one-to-one correspondence with brackets $[\,,\,]\colon V\ot V\to V$ making $V$ a Lie algebra (respectively, superalgebra).
For arbitrary $c$, we need the following generalization of Lie algebra introduced by Gurevich \cite{Gu} under the name ``Lie $c$-algebra''.

\begin{defn}
Let $L$ be a vector space, $c\colon L\ot L\to L\ot L$ a symmetric braiding, and $[\,,\,]\colon L\ot L\to L$ a linear map.
Then $(L,[\,,\,],c)$ is a {\em braided Lie algebra} if
\begin{align*}
&c([\,,\,]\ot\id_L)=(\id_L\ot[\,,\,])(c\ot\id_L)(\id_L\ot c)&\mbox{(compatibility),}\\
&[\,,\,](\id_{L\ot L}+c)=0&\mbox{(anticommutativity)}\\
&[\,,\,]([\,,\,]\ot \id_L)\Big(\id_{L\ot L\ot L}+(c\ot\id_L)(\id_L\ot c)+(c\ot\id_L)(\id_L\ot c)\Big)=0&\mbox{(Jacobi identity).}
\end{align*}
\end{defn}

Note that the compatibility condition (together with $c^2=\id$) simply means that the bracket commutes with $c$, and the above Jacobi identity implies a
similar identity for $[\,,\,](\id_L\ot[\,,\,])$ instead of $[\,,\,]([\,,\,]\ot\id_L)$. It is straightforward to check that if a vector space $A$ is equipped with
a symmetric braiding $c$ and an associative product $m\colon A\ot A\to A$ that commutes with $c$ then $(A,[\,,\,]_c,c)$ is a braided Lie algebra,
where $[\,,\,]_c$ is the {\em braided commutator} $m(\id_{A\ot A}-c)$.

Braided Lie algebras naturally arise as Lie algebras in a symmetric tensor category $\V$. A Lie algebra in $\V$ is an object $L$ endowed with a morphism
$[\,,\,]\colon L\ot L\to L$ such that the anticommutativity and Jacobi identity hold for $c=c_{L,L}$. If $(H,\beta)$ is a
{\em cotriangular} bialgebra (i.e., a CQT bialgebra satisfying $\beta^{-1}(h,k)=\beta(k,h)$ for all $h,k\in H$) then the category $\M^H$ is symmetric;
Lie algebras in this category were introduced and studied in \cite{BFM1,BFM2} under the name {\em $(H,\beta)$-Lie algebras}.
By an argument similar to \cite{Tk} (see Subsection \ref{ss:braided_bialg} above), any finite-dimensional braided Lie algebra can be regarded
as an $(H,\beta)$-Lie algebra for a suitable cotriangular bialgebra (Hopf algebra if the braiding is rigid).

Given a braided Lie algebra $(L,[\,,\,],c)$, the {\em universal enveloping algebra}, which we will denote $\cU_c(L)$, is the quotient
of the tensor algebra $T(L)$ by the ideal generated by the degree $2$ elements $x\ot y-c(x\ot y)-[x,y]$ where $x,y\in L$.
The usual increasing filtration of $T(L)$ gives rise to the {\em standard filtration} of $\cU_c(L)$.
As one would expect, $\cU_c(L)$ becomes a braided bialgebra if we declare the elements of $L$ primitive.
It is not true in general that, given an ordered basis of $L$, the corresponding PBW monomials form a basis of $\cU_c(L)$.
However, the following version of PBW Theorem holds.

\begin{thm}{\cite[Theorem 7.1]{Kharch}}
The graded algebra $\gr \cU_c(L)$ associated to the standard filtration of
$\cU_c(L)$ is naturally isomorphic to $\cU_c(L^\circ)$ where $L^\circ$ denotes
the braided Lie algebra with the same underlying braided vector space as $L$ but with zero bracket.\qed
\end{thm}

The standard filtration of $\cU_c(L)$ coincides with its coradical filtration. Also $\cU_c(L^\circ)=\cB(L,c)$.

It follows that graded deformations of $\cB(V,c)$ as a braided augmented algebra or as a braided bialgebra (with a fixed braiding) are in one-to-one
correspondence with brackets on $V$ making it a braided Lie algebra. Here the ``graded deformations'' and ``braided Lie algebras'' can be understood in the sense
of a stand-alone object or an object in $\M^H$ for a suitable cotriangular bialgebra $(H,\beta)$.

For $H=\kk G$, where $G$ is an abelian group, the cotriangular structures on $H$ are linear extensions of skew-symmetric bicharacters
$\beta\colon G\times G\to\kk^\times$. In this case the $(H,\beta)$-Lie algebras are known as the {\em color Lie superalgebras with grading group $G$
and commutation factor $\beta$}. Note that the braiding is diagonal and, conversely, any braided Lie algebra with a diagonal braiding can be regarded
as a color Lie superalgebra for some $G$ and $\beta$.

By a trick going back to Scheunert \cite{Sch}, color Lie superalgebras can be twisted to become ordinary Lie superalgebras.
This procedure works in the same way for all color Lie superalgebras with given $G$ and $\beta$, and is associated to a suitable cocycle twist of
$(\kk G,\beta)$ as a CQT bialgebra. Recall that a {\em right 2-cocycle} on a bialgebra $H$ is a convolution-invertible map
$\sigma\colon H\ot H\to\kk$
satisfying the following equations for all $h,k,\ell\in H$:
\begin{align*}
\sigma(h,k_{(1)}\ell_{(1)})\sigma(k_{(2)},\ell_{(2)})&=\sigma(h_{(1)}k_{(1)},\ell)\sigma(h_{(2)},k_{(2)}), &
\sigma(h,1)&=\sigma(1,h)=\varepsilon(h).
\end{align*}
Also recall that if $(H,\beta)$ is a cotriangular (more generally, CQT) bialgebra then $(H_\sigma,\beta_\sigma)$ is again a cotriangular (respectively, CQT)
bialgebra, see e.g. \cite{KS}; here $H_\sigma=H$ as a coalgebra, the multiplication of $H_\sigma$ is given by
\begin{equation*}
h\cdot_\sigma k=\sigma^{-1}(h_{(1)},k_{(1)})h_{(2)}k_{(2)}\sigma(h_{(3)},k_{(3)}),
\end{equation*}
and
\begin{equation*}
\beta_\sigma(h,k)=\sigma^{-1}(k_{(1)},h_{(1)})\beta(h_{(2)}k_{(2)})\sigma(h_{(3)},k_{(3)}).
\end{equation*}
Moreover, $\sigma$ yields an equivalence of braided tensor categories $\M^H$ and $\M^{H_\sigma}$, which is the identity on objects and morphisms
and only transforms the tensor product. If $A$ is an algebra (not necessarily associative) in $\M^H$ with
multiplication $m\colon A\ot A\to A$, then the corresponding algebra in $\M^{H_\sigma}$ is $A$ as an $H$-comodule but with new multiplication:
\begin{equation*}
m_\sigma(a\ot b)=\sigma(a_{(1)},b_{(1)})m(a_{(0)}\ot b_{(0)}).
\end{equation*}
We denote this new algebra by $A_\sigma$ and call it the $\sigma$-{\em twist} of $A$. It is shown in \cite{Ko} that multilinear polynomial identities of $A$
are preserved under $\sigma$-twist if we interpret them in each of the categories $\M^H$ and $\M^{H_\sigma}$ in terms of the appropriate action of symmetric groups
on tensor powers of $A$. In particular, associative algebras remain associative and $(H,\beta)$-Lie algebras become $(H_\sigma,\beta_\sigma)$-Lie algebras.

If $H$ is cocommutative then $H_\sigma=H$ but $\beta$ is twisted. If $H=\kk G$, with $G$ an abelian group, then there exists a $2$-cocycle
$\sigma\colon G\times G\to\kk^\times$ such that $\beta_\sigma$ is a ``sign bicharacter'':
\[
\beta_\sigma(g,h)=\left\{\begin{array}{lr}-1 & \mbox{ if }g,h\in G_-,\\ 1 & \mbox{otherwise;}\end{array}\right.
\]
where $G_-=G\setminus G_+$ and $G_+$ is a subgroup of index $\leq 2$. It follows that $\sigma$ twists any color Lie superalgebra
$L$ with commutation factor $\beta$ into a Lie superalgebra with even part $L_+$ and odd part $L_-$, where $L_\pm=\bigoplus_{g\in G_\pm}L_g$.

Etingof and Gelaki \cite{EG1} showed that, under a certain condition on the antipode called {\em pseudo-involutivity},
a cotriangular Hopf algebra $(H,\beta)$ can be twisted by a suitable cocycle to become the algebra of regular functions on a pro-algebraic group $G$ such that
$\beta_\sigma=\frac12(\ep\ot\ep+\ep\ot a+a\ot\ep-a\ot a)$ for some central element $a\in G$ with $a^2=1$.
It immediately follows \cite[Theorem 4.3]{Ko} that the same cocycle twists $(H,\beta)$-Lie algebras to Lie superalgebras equipped with a $G$-action.
Here the even and odd components are just the eigenspaces with respect to the action of $a$, with eigenvalues $1$ and $-1$ respectively.

If $H$ is finite-dimensional then pseudo-involutivity of the antipode is equivalent to involutivity and hence to semisimplicity of $H$.
Later, Etingof and Gelaki \cite{EG2,Ge} described all finite-dimensional cotriangular Hopf algebras by showing that $(H,\beta)$ can be twisted in such a way that
its dual triangular Hopf algebra becomes a ``modified supergroup algebra''. As a corollary, any $(H,\beta)$-Lie algebra is twisted to a Lie superalgebra
equipped with a supergroup action \cite[Theorem 4.6]{Ko}.

One can use the twisting procedure to transfer known properties of Lie superalgebras to $(H,\beta)$-Lie algebras in the above cases. Let $\cU_\beta(L)$ be
the universal enveloping algebra of an $(H,\beta)$-Lie algebra $L$, i.e., $\cU_c(L)$ for $c=c_{L,L}$ determined by $\beta$. It is straightforward to verify that
$\cU_{\beta_\sigma}(L_\sigma)$ is naturally isomorphic to $(\cU_\beta(L))_\sigma$. In particular, for $V$ in $\M^H$ and $c=c_{V,V}$ induced by $\beta$,
the $\sigma$-twist of the Nichols algebra $\cB(V,c)$ is naturally isomorphic to $\cB(V,c')$ where $c'$ is the braiding on $V$ induced by $\beta_\sigma$. This
gives an alternative proof of PBW Theorem for $(H,\beta)$-Lie algebras \cite{Ko}.

\begin{thm}\label{thm:symmetric}
Let $(H,\beta)$ be a cotriangular Hopf algebra that is either pseudo-involutive or finite-dimensional.
Let $V$ be a finite-dimensional $H$-comodule with the corresponding braiding $c$. If the Nichols algebra $\cB(V,c)$ is finite-dimensional
then it does not admit nontrivial graded deformations as an augmented algebra or bialgebra in $\M^H$.
\end{thm}

\begin{proof}
By our assumption on $(H,\beta)$, there exists a cocycle $\sigma$ such that $(H_\sigma,\beta_\sigma)$ is as described by Etingof and Gelaki.
Then the braiding $c'$ induced by $\beta_\sigma$ on $V$ is just the signed flip associated to a $\Z_2$-grading $V=V_0\oplus V_1$,
so $\cB(V,c')=S(V_0)\ot\Lambda(V_1)$, which is finite-dimensional only if $V_0=0$.
But in this case $V$ does not admit nontrivial Lie superalgebra structures. It follows that $V$ does not admit nontrivial $(H,\beta)$-Lie algebra structures
and hence $\cB(V,c)$ is rigid in $\M^H$.
\end{proof}

\begin{cor}
Let $(V,c)$ be a finite-dimensional braided vector space such that $c$ can be obtained from a coaction by a finite-dimensional cotriangular Hopf algebra.
If $\cB(V,c)$ is finite-dimensional then it does not admit nontrivial graded deformations as a braided augmented algebra or bialgebra.
\end{cor}

\begin{proof}
By assumption, $V$ can be regarded as an object in $\M^H$ for some finite-dimensional cotriangular Hopf algebra $(H,\beta)$ such that $c=c_{V,V}$.
Any graded deformation of $\cB(V,c)$ can be realized in $\M^{\bar{H}}$ for some quotient $(\bar{H},\bar{\beta})$ of the cotriangular Hopf algebra $(H,\beta)$,
so it must be trivial by the above theorem.
\end{proof}

\section[The vanishing of second cohomology]{The vanishing of second algebra cohomology for a class\\ of augmented algebras in a braided category}

Let $\V$ be a braided tensor category consisting of vector spaces and linear maps.
Let $(\cB,\ep)$ be an augmented algebra in $\V$ acting trivially (i.e., via $\ep$) on some $U$ in $\V$.

\smallskip
\noindent$\diamond$ A  map $f\colon \cB\otimes \cB\to U$ in $\V$ is an $\ep$-\emph{cocycle} if $f(1,a)=0=f(a,1)$ and $f(xy,z)=f(x,yz)$ for all $a\in \cB$ and all $x,y,z\in \cB^+$.  The space of all $\ep$-cocycles is denoted by ${\rm Z}^2_\ep(\cB, U)$.

\smallskip
\noindent$\diamond$ An $\ep$-cocycle is an $\ep$-\emph{coboundary} if there exists a map $t\colon \cB\to U$ such that $t(1)=0$ and $f(x,y)=t(xy)$ for all $x,y\in \cB^+$.  The space of all $\ep$-coboundaries is denoted by ${\rm B}^2_\ep(\cB,U)$.

\smallskip
\noindent$\diamond$ The quotient of $\ep$-cocycles by $\ep$-coboundaries is denoted by ${\rm H}^2_\ep(\cB,U)={\rm Z}^2_\ep(\cB,U)/{\rm B}^2_\ep(\cB,U)$.

\smallskip
In what follows $(\cB^+)^2$ denotes the range of the multiplication $\cB^+\otimes_\cB \cB^+\stackrel{\m}{\to} \cB^+$, i.e., $(\cB^+)^2=\mathrm{span}\{xy\;|\;x,y\in \cB^+\}$.

\begin{lemma}[cf. {\cite[Subsection 4.1]{MW}}] Let $\cB$ be an augmented algebra in $\V$ and let $M=\ker\left( \cB^+\otimes_\cB \cB^+\stackrel{\m}{\to} \cB\right)$.
If the map $\cB^+\otimes_\cB \cB^+\stackrel{\m}{\to} (\cB^+)^2$ splits in $\V$, then for every space $U\in\V$, we have
${\rm H}^2_\ep(\cB,U) = \Hom(M, U)$.
\end{lemma}

\begin{proof}
Let $\varphi\colon (\cB^+)^2\to \cB^+\otimes_\cB \cB^+$ be a splitting of $\m$ and let $p\colon \cB^+\ot \cB^+\to \cB^+\ot_\cB \cB^+$ be the canonical projection.
We define a map $\Phi\colon \Hom(M, U)\to {\rm H}_\ep^2(\cB,U) $ as follows: if $f\colon M\to U$, then the cocycle $\Phi(f)\colon \cB^+\ot \cB^+\to U$ is
$\Phi(f) = f(p - \varphi\m)$. The inverse $\Psi$ of $\Phi$ is defined as follows: if $g\colon \cB^+\otimes \cB^+\to U$ is a cocycle,
then $\Psi(g)\colon M\to U$ is the unique map such that $\Psi(g)p=g$.  Now observe that maps $\Phi$ and $\Psi$ are well defined:
$\Phi(f)$ is always a cocycle and $\Psi(g)=0$ whenever $g$ is a coboundary.  Note also that $\Psi\Phi=\id$ and that
the range of $\Phi\Psi-\id$ consists of coboundaries.
\end{proof}

\begin{remark}\label{split} 
A splitting of $\cB^+\otimes_\cB \cB^+\stackrel{\m}{\to} (\cB^+)^2$ in $\V$ automatically exists (it is usually not unique) if $\cB^+\otimes_\cB \cB^+$ is a semisimple object in $\V$.  This happens whenever $\V$ is either the category of Yetter-Drinfeld modules over a semisimple and cosemisimple Hopf algebra or the category of comodules over a cosemisimple CQT bialgebra.  It also happens if $\V$ is the category of Yetter-Drinfeld modules over $\mathbb{k}\Gamma$, where $\Gamma$ is a possibly infinite abelian group, and $\mathcal{B}$ is a direct sum of its one-dimensional subobjects in $\V$ (e.g., a quotient of the tensor algebra $T(V)$, for some $V$ of finite dimension over $\mathbb{k}$).
\end{remark}

Let $V$ be a an object in $\V$, $T(V)$ its tensor algebra and $I$ an ideal generated
by homogeneous elements of degree at least two. Let $\cB=T(V)/I$ and let $\pi\colon T(V)\to \cB$ be the canonical projection.
We also abbreviate $T(V)^+=\bigoplus_{n\ge 1} V^{\ot n}$ and
$T(V)_{(2)}=\bigoplus_{n\ge 2} V^{\ot n}$.

\begin{lem} The following is a commutative diagram:
\[
\begin{CD}
                         @.  I\ot T(V)^+ +T(V)^+\ot I @>{\m}>> I\\
 @.                             @VVV                    @VVV \\
T(V)^+\ot T(V)\ot T(V)^+ @>{\id\ot \m-\m\ot\id}>> T(V)^+\ot T(V)^+  @>\m>> T(V)_{(2)} \\
@V{\pi\ot\pi\ot\pi}VV                             @V{\pi\ot\pi}VV                   @V{\widetilde{\pi\ot\pi}}VV \\
\cB^+\ot \cB\ot \cB^+          @>{\id\ot \m-\m\ot\id}>>   \cB^+\ot \cB^+      @>p>>  \cB^+\ot_\cB \cB^+\\
 @.                                @V{\m}VV                @V\widetilde{\m}VV\\
                         @.       (\cB^+)^2        @=     (\cB^+)^2
\end{CD}
\]
where the maps $\widetilde{\m}$ and $\widetilde{\pi\ot\pi}$ are the universal maps arising from fact (1) below. Moreover, we have the following facts:
\begin{enumerate}
\item The second and third rows of the diagram are cokernel diagrams.
\item The second column of the diagram is exact at $T(V)^+\ot T(V)^+$.
\item The composition $T(V)_{(2)}\stackrel{\widetilde{\pi\ot\pi}}{\to} \cB^+\ot_\cB \cB^+\stackrel{\widetilde{\m}}{\to} (\cB^+)^2$
is equal to the restriction of $\pi$ to $T(V)_{(2)}$.
\item The map $\widetilde{\pi\ot\pi}$ is surjective.
\item If $\varphi\colon T(V)_{(2)}\to T(V)^+\ot T(V)^+$ is any splitting of multiplication
(e.g., the composition $T(V)_{(2)}\stackrel{\sim}{\to} V\ot T(V)^+\to T(V)^+\ot T(V)^+$ is such a splitting), then $\widetilde{\pi\ot\pi}=p(\pi\ot\pi)\varphi$.
\end{enumerate}
\end{lem}
\begin{proof}
Clearly, each of the squares of the diagram commutes. We prove the remaining claims below:
\begin{enumerate}
\item The third row is a cokernel diagram by definition. The second row is a cokernel diagram due to the fact that $T(V)^+=V\ot T(V)$ as a right
$T(V)$-module (with the obvious action on the second tensor factor), hence $T(V)^+\ot_{T(V)} T(V)^+= V\ot T(V)^+$, and
$V\ot T(V)^+\stackrel{\m}{\to} T(V)_{(2)}$ is an isomorphism.
\item Clear.
\item As $\pi$ is an algebra map, we have $\m(\pi\ot\pi)\m= \pi \m$. Hence $\widetilde{\m}(\widetilde{\pi\ot\pi})\m = \pi \m$.
By the universal property of cokernels this means that
$\widetilde{\m}(\widetilde{\pi\ot\pi})=\pi$.
\item Follows from the fact that maps $p$ and $\pi\ot\pi$ are surjective.
\item Follows from the universal property of cokernels.
\end{enumerate}
\end{proof}

\begin{cor}
The following sequence is exact:
\begin{equation*}
0\to {T(V)^+}I+I T(V)^+\to I\stackrel{\widetilde{\pi\ot\pi}}{\longrightarrow} \cB^+\ot_\cB \cB^+ \stackrel{\widetilde{\m}}{\to} (\cB^+)^2\to 0
\end{equation*}
Therefore, $I/(T(V)^+ I+IT(V)^+)\simeq \ker\left( \cB^+\ot_\cB  \cB^+\to (\cB^+)^2\right) $.
\end{cor}
\begin{proof}
To avoid ambiguity, we denote the restriction of $\widetilde{\pi\ot\pi}$ to $I$ by $\tau$.
We first prove that $\ker(\tau)= {T(V)^+}I+I T(V)^+$.  The inclusion
${T(V)^+}I+I T(V)^+\subseteq \ker(\widetilde{\pi\ot\pi})$ follows from
$\widetilde{\pi\ot\pi}({T(V)^+}I+I T(V)^+)=(\widetilde{\pi\ot\pi})\m({T(V)^+}\ot I+I\ot T(V)^+)=
p(\pi\ot\pi)({T(V)^+}\ot I+I\ot T(V)^+)=0.$

Let $x\in \ker(\tau)$.
Since $\m(T(V)^+\ot T(V)^+)=T(V)_{(2)}$, there exists $y\in T(V)^+\ot T(V)^+$ such that $\m(y)=x$.
Now $0=(\widetilde{\pi\ot\pi})\m(y)=p(\pi\ot\pi)(y)$, and hence $(\pi\ot\pi)y=(\id\ot \m-\m\ot\id)z$ for some $z\in \cB^+\ot \cB\ot \cB^+$.  Let $w\in T(V)^+\ot T(V)\ot T(V)^+$ be such that $(\pi\ot\pi\ot\pi)(w)=z$.  Define $y'=y-(\id\ot \m-\m\ot\id)w$. As $(\pi\ot\pi)y'=0$ we have that $y'\in I\ot {T(V)^+} + {T(V)^+}\ot I$ and hence $x=\m(y)=\m(y')\in I T(V)^+ + T(V)^+ I$.

We now prove that $\widetilde{\pi\ot\pi}(I)=\ker(\cB^+\ot_\cB \cB^+\stackrel{\widetilde{\m}}{\to} (\cB^+)^2)$.  The inclusion $\subseteq$ follows from part (3) of the lemma above: $\widetilde{\m}\widetilde{(\pi\ot\pi)}(I)=\pi(I)=0$.  The inclusion $\supseteq$ follows from the fact that $\widetilde{\pi\ot\pi}$ is surjective.
\end{proof}

\begin{cor}
If $I$ is generated by a subobject $R$, then the induced morphism
\[
R\to \ker\left( \cB^+\ot_\cB \cB^+\to (\cB^+)^2\right)
\]
is surjective.\qed
\end{cor}

We summarize the above results in a theorem which will be needed in the next section to establish rigidity of certain graded bialgebras in $\V$.

\begin{thm}\label{thm-one}
Let $V$ be a an object in $\V$ and $T(V)$ its tensor algebra. Let $R\subset T(V)_{(2)}$ be a graded subspace that is an object in $\V$.
Consider the augmented algebra $\cB=T(V)/\langle R\rangle$ and an object $U$ in $\V$ on which $\cB$ acts trivially (i.e., via $\varepsilon$).
If the multiplication map $\cB^+\otimes_\cB \cB^+\stackrel{\m}{\to} (\cB^+)^2$ splits in $\V$, then there is an injection ${\rm H}_\ep^2(\cB, U)\to \Hom(R, U)$.

In particular, if $f$ is an $\varepsilon$-cocycle such that for every $u\in \cB\otimes \cB$ in the range of the composition
$R\to V\otimes T(V)^+\to \cB\otimes \cB$ we have $f(u)=0$, then $f$ is an $\varepsilon$-coboundary.\qed
\end{thm}

\section[A sufficient condition for rigidity]{A sufficient condition for rigidity of graded bialgebras\\ in a braided category}

Let $\cB$ be a graded bialgebra in $\V$.  For a homogeneous map $f\colon \cB\ot \cB\to \cB$ of degree $\ell$ and a nonnegative integer $r$ we define $f_r\colon \cB\ot \cB\to \cB$ by $f_r|_{(\cB\ot \cB)_r}=f$ and $f_r|_{(\cB\ot \cB)_s}=0$ for $s\not=r$.  For $g\colon \cB\to \cB\ot \cB$, we define $g_r$ analogously.  We also define $f_{\le r}$ by $f_{\le r}=\sum_{i=0}^r f_i$, and $f_{<r}, g_{\le r}, g_{<r}$ in a similar fashion.

\begin{lemma}[cf. {\cite[Lemma 2.3.6]{MW}}]\label{lem-one}  Let $\cB$ be a graded bialgebra in $\V$ such that $\cB_0=\kk$
and $\cB$ is generated as an algebra by $\cB_1$.
\begin{enumerate}
\item
If $(f,g)\in {\rm Z}^2_{\mathrm{b}}(\cB)_\ell$, $r>1$, $f_{\le r}=0$, and $g_{<r}=0$, then $g_r=0$.
\item
If $(f,g)\in {\rm Z}^2_{\mathrm{b}}(\cB)_\ell$, $\ell<0$, and $f_{\le r}=0$, then $g_{\le r}=0$.
\item
If $(0,g)\in {\rm Z}^2_{\mathrm{b}}(\cB)_\ell$, $\ell<0$, then $g=0$.
\end{enumerate}
\end{lemma}
\begin{proof}
The proof in \cite{MW} carries over word for word.  First note that  for every $(f,g)\in {\rm Z}^2_{\mathrm{b}}(\cB)$ we have $f_{\le 1}=0$ and $g_{\le 2}=0$, due to the fact that $(\cB^+\ot \cB^+)_0=0=(\cB^+\ot \cB^+)_1$.  Hence (1) easily yields (2) and (3).

For (1) recall that $\na^c f = -\na^h g$ by Equation \eqref{eq-compat}.  If $r>1$, $a\in \cB_1$ and $b\in \cB_{r-1}$, then
$$ (\na^c f)(a,b)=0=-(\na ^h g)(a,b) = -(\De a)g(b) + g(ab) - g(a)(\De b) = g(ab).$$
As $\cB_r$ is spanned by such products $ab$, we have that $g(\cB_r)=0$.
\end{proof}

\begin{lemma}[cf. {\cite[Lemma 2.3.5]{MW}}]  Let $\cB$ be a connected graded bialgebra in $\V$,
let $r\in \N$, and let $f\colon \cB\otimes \cB\to \cB$ be a homogeneous unital Hochschild cocycle in $\V$ (with respect to left and right regular actions of $\cB$ on itself).  If $f_{<r}=0$, then $f_r\colon \cB\ot \cB\to \cB$ is an $\ep$-cocycle.
\end{lemma}
\begin{proof}
This follows directly from $\na^h f =0$, see Equation \eqref{eq-assoc}.
\end{proof}

\begin{thm}[cf. {\cite[Lemma 4.2.2]{MW}}]\label{thm:suff_cond}
Let $V$ be a an object in $\V$ and $T(V)$ its (braided) tensor bialgebra. Let $R\subset T(V)_{(2)}$ be a graded subspace that is an object in $\V$ and generates a biideal in $T(V)$.
Consider the quotient $\cB=T(V)/\langle R\rangle$, which is a graded bialgebra in $\mathcal{V}$, and assume that the multiplication map
$\m\colon \cB^+\otimes_\cB \cB^+\to (\cB^+)^2$ splits in $\mathcal{V}$.  If for some negative $\ell$ we have that $\Hom(R, P(\cB))_\ell = 0$, then $\widehat{\coh}^2_{\mathrm{b}}(\cB)_\ell=0$.

In particular, if $\Hom(R,P(\cB))_\ell=0$ for all negative $\ell$, then $\cB$ is rigid.
\end{thm}
\begin{proof}
Let $(f,g)\in {\rm Z}^2_{\mathrm{b}}(\cB)_\ell$.  We will find a map $s=\sum_{r=0}^\infty s_r\colon \cB\to \cB$  such that for every nonnegative $r$, $(f,g)_r=(\na^h s_r, -\na^c s_r)$, from where the result trivially follows since $(f,g)=\na^b \sum_{r=0}^\infty s_r$.
Here the sum $s=\sum_{r=0}^\infty s_r$ is potentially infinite but locally finite. The cases $r=0,1$ are clear. Suppose that $s_0,\ldots s_{r-1}$ have been found.
Let $(f',g') = (f,g)-\na^b s_{<r} = (f,g)-\sum_{i=0}^{r-1} (\na^h s_i, -\na^c s_i)$.  Note that, by assumption, $f'|_{<r} =0$ and hence,
by Lemma \ref{lem-one}, also $g'|_{<r}=0$.  Let $u\in (\cB\otimes \cB)$ be
in the range of the composition $R\to V\otimes T(V)^+\to \cB\otimes \cB$.  Since $\m(u)=0$ we have from Equation \eqref{eq-compat} that
$f_r(u)\in P(\cB)$.  Therefore, the composition $R\to V\otimes T(V)^+\to \cB\otimes \cB\stackrel{f}{\to} \cB$ has range in $P(\cB)$ and
must be the zero map.  By Theorem \ref{thm-one}, we get a map $t\colon \cB\to \cB$ such that $f_r=t\m$.  Now define $s_r=t_r$ and observe that
$f'_{\le r} = f'_r = \na^h s_r$. Hence, by Lemma \ref{lem-one}, we also have $g'_r = -\na^c s_r$.
\end{proof}


\section{Nichols algebras of diagonal type}\label{s:diagonal}

In what follows $(V,c)$ will denote a braided vector space of diagonal type, $\dim V=\theta$, such that the associated Nichols algebra $\cB(V)$
has a finite root system $\Delta_+^V$ in the sense of \cite{H-classif}, i.e., $\Delta_+^V$ is the set of $\N_0^\theta$-degrees of generators of a PBW basis.
In particular, this is the case if $\cB(V)$ is finite-dimensional.
Let
\begin{align}
-c_{ij}^{V}&:= \min \left\{ n \in \N_0\;|\; (n+1)_{q_{ii}}
(1-q_{ii}^n q_{ij}q_{ji} )=0 \right\}, & j & \neq i. \label{defn:mij}
\end{align}
Now we fix
\begin{itemize}
  \item a basis $\{x_1,\ldots, x_\theta\}$ of $V$ and $q_{ij}\in\kk^\times$ such that $c(x_i\otimes x_j)=q_{ij}x_j\otimes x_i$,
  \item elements $x_\alpha\in\cB(V)$ of degree $\alpha$, $\alpha\in \Delta_+^V$, which generate a PBW basis, see \cite{Ang-crelle}.
\end{itemize}
We use the following notation:
\begin{itemize}
  \item $\widetilde{q_{ij}}:=q_{ij}q_{ji}$ for all $i\neq j$.
  \item $\chi\colon\Z^\theta\times\Z^\theta\to\kk^\times$ is the bicharacter such that $\chi(\alpha_i,\alpha_j)=q_{ij}$, $1\le i,j\le\theta$, where 
  $\{\alpha_1,\ldots,\alpha_\theta\}$ is the canonical basis of $\Z^\theta$.
  \item $N_\alpha$ is the order of $q_\alpha:=\chi(\alpha,\alpha)$, $\alpha\in\Delta_+^V$.
  \item $\G_N$ is the group of roots of unity of order $N$ and $\G_N'$ is the subset of primitive roots of unity of order $N$, $N\in\N$.
  \item $\cO(V)$ is the set of Cartan roots of $V$, i.e., the orbit of Cartan vertices under the action of the Weyl groupoid. Recall that $i\in\{1,\dots,\theta\}$ is a Cartan vertex of $V$ if $\widetilde{q_{ij}}=q_{ii}^{c_{ij}^{V}}$ for all $j \neq i$ \cite[Definition 2.6]{Ang-crelle}.
\end{itemize}
We recall the following result, which gives a presentation by generators and relations for any Nichols algebra of diagonal type with finite root system.

\begin{thm}\label{thm:presentacion minima}\cite{Ang-crelle}
$\cB(V)$ is presented by generators $x_1,\ldots,x_\theta$ and relations:
\begin{align}
&x_\alpha^{N_\alpha}, &\alpha \in \cO(V);\label{eqn:potencia raices}
\\ &(\ad_cx_i)^{1-c_{ij}^V}x_j, & q_{ii}^{1-c_{ij}^V} \neq 1;\label{eqn:relacion quantum Serre}
\\ &x_i^{N_i}, & i \mbox{ is not a Cartan vertex};\label{eqn:potencia raices simples}
\end{align}

\noindent $\diamond$ if $i,j \in \{1, \ldots, \theta \}$ satisfy
$q_{ii}=\widetilde{q_{ij}}=q_{jj}=-1$, and there exists $k\neq i,j$
such that $\widetilde{q_{ik}}^2\neq1$ or
$\widetilde{q_{jk}}^2\neq1$,
\begin{equation}\label{eqn:relacion dos vertices con -1}
x_{ij}^2;
\end{equation}

\noindent $\diamond$ if $i,j,k \in \{1, \ldots, \theta \}$ satisfy
$q_{jj}=-1$,
$\widetilde{q_{ik}}=\widetilde{q_{ij}}\widetilde{q_{kj}}=1$,
$\widetilde{q_{ij}}\neq -1$,
\begin{equation}\label{eqn:relacion vertice -1}
\left[ x_{ijk} , x_j \right]_c;
\end{equation}

\noindent $\diamond$ if $i,j \in \{1, \ldots, \theta \}$ satisfy
$q_{jj}=-1$, $q_{ii}\widetilde{q_{ij}}\in \G'_6$,
$\widetilde{q_{ij}}\neq -1$, and also $q_{ii}\in \G'_3$ or
$-c_{ij}^V\geq 3$,
\begin{equation}\label{eqn:relacion estandar B2}
\left[ x_{iij}, x_{ij} \right]_c;
\end{equation}

\noindent $\diamond$ if $i,j,k \in \{1, \ldots, \theta \}$ satisfy $q_{ii}=\pm \widetilde{q_{ij}}\in\G'_3$, $\widetilde{q_{ik}}=1$, and also
$-q_{jj}=\widetilde{q_{ij}}\widetilde{q_{jk}}=1$ or $q_{jj}^{-1}=\widetilde{q_{ij}}=\widetilde{q_{jk}}\neq -1$,
\begin{equation}\label{eqn:relacion estandar B3}
\left[ x_{iijk} , x_{ij} \right]_c;
\end{equation}

\noindent $\diamond$ if $i,j,k \in \{1, \ldots, \theta \}$ satisfy $\widetilde{q_{ik}}, \widetilde{q_{ij}}, \widetilde{q_{jk}} \neq 1$,
\begin{equation}\label{eqn:relacion triangulo}
x_{ijk}-\frac{1-\widetilde{q_{jk}}}{q_{kj}(1-\widetilde{q_{ik}})}\left[x_{ik},x_j\right]_c-q_{ij}(1-\widetilde{q_{jk}}) \ x_jx_{ik};
\end{equation}

\noindent $\diamond$ if $i,j,k\in\{1,\ldots,\theta\}$ satisfy one of the following situations

\vi $q_{ii}=q_{jj}=-1$, $\widetilde{q_{ij}}^2= \widetilde{q_{jk}}^{-1}$, $\widetilde{q_{ik}}=1$, or

\vii $\widetilde{q_{ij}}=q_{jj}=-1$, $q_{ii}= -\widetilde{q_{jk}}^2\in\G'_3$, $\widetilde{q_{ik}}=1$, or

\viii $q_{kk}=\widetilde{q_{jk}}=q_{jj}=-1$, $q_{ii}= -\widetilde{q_{ij}}\in\G'_3$, $\widetilde{q_{ik}}=1$, or

\viv $q_{jj}=-1$, $\widetilde{q_{ij}}=q_{ii}^{-2}$, $\widetilde{q_{jk}}=-q_{ii}^3$, $\widetilde{q_{ik}}=1$, or

\vv $q_{ii}=q_{jj}=q_{kk}=-1$, $\pm\widetilde{q_{ij}}=\widetilde{q_{jk}}\in\G'_3$, $\widetilde{q_{ik}}=1$,

\begin{equation}\label{eqn:relacion super C3}
\left[ \left[x_{ij}, x_{ijk} \right]_c, x_j \right]_c;
\end{equation}

\noindent $\diamond$ if $i,j,k\in\{1,\ldots,\theta\}$ satisfy $q_{ii}=q_{jj}=-1$, $\widetilde{q_{ij}}^3=\widetilde{q_{jk}}^{-1}$, $\widetilde{q_{ik}}=1$,
\begin{equation}\label{eqn:relacion super G3}
\left[ \left[x_{ij}, \left[x_{ij}, x_{ijk} \right]_c \right]_c, x_j \right]_c;
\end{equation}

\noindent $\diamond$ if $i,j,k\in\{1,\ldots,\theta\}$ satisfy
$q_{jj}=\widetilde{q_{ij}}^2=\widetilde{q_{jk}}\in \G'_3$,
$\widetilde{q_{ik}}=1$,
\begin{equation}\label{eqn:relacion super C3 raiz de orden 3}
\left[ \left[ x_{ijk} , x_j \right]_c x_j \right]_c;
\end{equation}

\noindent $\diamond$ if $i,j,k\in\{1,\ldots,\theta\}$ satisfy
$q_{kk}=q_{jj}=\widetilde{q_{ij}}^{-1}=\widetilde{q_{jk}}^{-1}\in
\G'_9$, $\widetilde{q_{ik}}=1$, $q_{ii}=q_{kk}^6$
\begin{equation}\label{eqn:relacion fila 18 rango 3}
\left[ \left[ x_{iij} , x_{iijk} \right]_c, x_{ij} \right]_c;
\end{equation}

\noindent $\diamond$ if $i,j,k\in\{1,\ldots,\theta\}$ satisfy
$q_{ii}=\widetilde{q_{ij}}^{-1}\in \G'_9$,
$q_{jj}=\widetilde{q_{jk}}^{-1}=q_{ii}^5$, $\widetilde{q_{ik}}=1$,
$q_{kk}=q_{ii}^6$
\begin{equation}\label{eqn:relacion fila 18 rango 3, caso 2}
[\left[x_{ijk}, x_{j} \right]_c, x_k]_c -(1 +
\widetilde{q_{jk}})^{-1}q_{jk} \left[ \left[x_{ijk}, x_{k} \right]_c
, x_{j} \right]_c;
\end{equation}

\noindent $\diamond$ if $i,j,k\in\{1,\ldots,\theta\}$ satisfy
$q_{jj}=\widetilde{q_{ij}}^3=\widetilde{q_{jk}}\in \G'_4$,
$\widetilde{q_{ik}}=1$,
\begin{equation}\label{eqn:relacion super G3 raiz de orden 4}
\left[ \left[ \left[ x_{ijk} , x_j \right]_c, x_j \right]_c, x_j
\right]_c;
\end{equation}

\noindent $\diamond$ if $i,j,k\in\{1,\ldots,\theta\}$ satisfy
$q_{ii} = \widetilde{q_{ij}} =-1$, $q_{jj}= \widetilde{q_{jk}}^{-1}
\neq-1$, $\widetilde{q_{ik}}=1$,
\begin{equation}\label{eqn:relacion parecida a super C3}
\left[x_{ij}, x_{ijk} \right]_c;
\end{equation}

\noindent $\diamond$ if $i,j,k \in\{1,\ldots,\theta\}$ satisfy
$q_{ii}= q_{kk} =-1$, $\widetilde{q_{ik}}=1$, $\widetilde{q_{ij}}
\in \G'_3$, $q_{jj}= -\widetilde{q_{jk}} = \pm \widetilde{q_{ij}}$,
\begin{equation}\label{eqn:relacion parecida a super C3-bis}
[x_i, x_{jjk}]_c -(1 + q_{jj}^2)q_{kj}^{-1} \left[x_{ijk}, x_{j}
\right]_c - (1 + q_{jj}^2)(1 + q_{jj}) q_{ij} x_j x_{ijk};
\end{equation}

\noindent $\diamond$ if $i,j,k,l\in\{1,\ldots,\theta \}$ satisfy
$q_{jj}\widetilde{q_{ij}}= q_{jj}\widetilde{q_{jk}}=1$, $q_{kk}=-1$,
$\widetilde{q_{ik}}=\widetilde{q_{il}}=\widetilde{q_{jl}}=1$,
$\widetilde{q_{jk}}^2= \widetilde{q_{lk}}^{-1}= q_{ll}$,
\begin{equation}\label{eqn:relacion super C4}
\left[\left[\left[x_{ijkl},x_k\right]_c, x_j \right]_c, x_k \right]_c;
\end{equation}

\noindent $\diamond$ if $i,j,k,l\in\{1,\ldots,\theta \}$ satisfy
$\widetilde{q_{jk}}= \widetilde{q_{ij}}= q_{jj}^{-1}\in
\G'_4\cup\G'_6$, $q_{ii}=q_{kk}=-1$,
$\widetilde{q_{ik}}=\widetilde{q_{il}}=\widetilde{q_{jl}}=1$,
$\widetilde{q_{jk}}^3= \widetilde{q_{lk}}$,
\begin{equation}\label{eqn:relacion super C4 modificada}
\left[\left[x_{ijk},\left[x_{ijkl}, x_k \right]_c\right]_c, x_{jk}
\right]_c;
\end{equation}

\noindent $\diamond$ if $i,j,k,l\in\{1,\ldots,\theta \}$ satisfy
$q_{ll}=\widetilde{q_{lk}}^{-1}=
q_{kk}=\widetilde{q_{jk}}^{-1}=q^2$, $\widetilde{q_{ij}}=
q_{ii}^{-1}=q^3$ for some $q\in \kk^\times$, $q_{jj}=-1$,
$\widetilde{q_{ik}}=\widetilde{q_{il}}=\widetilde{q_{jl}}=1$,
\begin{equation}\label{eqn:relacion super F4-1}
\left[\left[\left[x_{ijk},x_j\right]_c, \left[x_{ijkl},x_j\right]_c
\right]_c, x_{jk} \right]_c;
\end{equation}

\noindent $\diamond$ if $i,j,k,l\in\{1,\ldots,\theta \}$ satisfy one
of the following situations

\vi $q_{kk}=-1$, $q_{ii}=\widetilde{q_{ij}}^{-1}= q_{jj}^2$,
$\widetilde{q_{kl}}= q_{ll}^{-1}= q_{jj}^3$, $\widetilde{q_{jk}}=
q_{jj}^{-1}$,
$\widetilde{q_{ik}}=\widetilde{q_{il}}=\widetilde{q_{jl}}=1$, or

\vii $q_{ii}=\widetilde{q_{ij}}^{-1}= -q_{ll}^{-1}=-\widetilde{q_{kl}}$,
$q_{jj}=\widetilde{q_{jk}}=q_{kk}=-1$,
$\widetilde{q_{ik}}=\widetilde{q_{il}}=\widetilde{q_{jl}}=1$,

\begin{equation}\label{eqn:relacion super F4-2}
\left[\left[x_{ijkl}, x_j \right]_c, x_k \right]_c-
q_{jk}(\widetilde{q_{ij}}^{-1}-q_{jj})
\left[\left[x_{ijkl},x_k\right]_c, x_j \right]_c;
\end{equation}

\noindent $\diamond$ if $i,j,k\in\{1,\ldots,\theta\}$ satisfy
$\widetilde{q_{jk}}=1$,
$q_{ii}=\widetilde{q_{ij}}=-\widetilde{q_{ik}}\in \G'_3$,
\begin{equation}
\left[x_i, \left[ x_{ij},x_{ik} \right]_c
\right]_c+q_{jk}q_{ik}q_{ji} \left[ x_{iik} ,x_{ij} \right]_c+q_{ij}
\, x_{ij} x_{iik};\label{eqn:relacion mji,mjk=2}
\end{equation}

\noindent $\diamond$ if $i,j,k\in\{1,\ldots,\theta\}$ satisfy
$q_{jj}=q_{kk}=\widetilde{q_{jk}}=-1$,
$q_{ii}=-\widetilde{q_{ij}}\in\G'_3$, $\widetilde{q_{ik}}=1$,
\begin{equation}\label{eqn:relacion especial rank3}
\left[x_{iijk}, x_{ijk} \right]_c;
\end{equation}

\noindent $\diamond$ if $i,j\in\{1,\ldots,\theta\}$ satisfy $-q_{ii}, -q_{jj}, q_{ii}\widetilde{q_{ij}}, q_{jj}\widetilde{q_{ij}} \neq 1$,
\begin{equation}\label{eqn:relacion mij,mji mayor que 1}
(1-\widetilde{q_{ij}})q_{jj}q_{ji}\left[x_i, \left[ x_{ij}, x_j \right]_c \right]_c - (1+q_{jj})(1-q_{jj}\widetilde{q_{ij}})x_{ij}^2;
\end{equation}

\noindent $\diamond$ if $i,j\in\{1,\ldots,\theta\}$ satisfy that $-c_{ij}^V\in \{4,5\}$, or $q_{jj}=-1$, $-c_{ij}^V=3$, $q_{ii} \in \G'_4$,
\begin{equation}\label{eqn:relacion mij mayor que dos, raiz alta}
\left[x_i,x_{3\alpha_i+2\alpha_j}\right]_c-\frac{1-q_{ii}\widetilde{q_{ij}}-q_{ii}^2\widetilde{q_{ij}}^2q_{jj}}{(1-q_{ii}\widetilde{q_{ij}})q_{ji}}
x_{iij}^2;
\end{equation}

\noindent $\diamond$ if $i,j\in\{1,\ldots,\theta\}$ satisfy $4\alpha_i+3\alpha_j\notin \Delta_+^V$, $q_{jj}=-1$ or $m_{ji}\geq2$, and also $-c_{ij}^V\geq 3$,
or $-c_{ij}^V=2$, $q_{ii}\in\G'_3$,
\begin{equation}\label{eqn:relacion (m+1)alpha i+m alpha j, caso 2}
x_{4\alpha_i+3\alpha_j}=[x_{3\alpha_i+2\alpha_j}, x_{ij} ]_c;
\end{equation}

\noindent $\diamond$ if $i,j\in\{1,\ldots,\theta\}$ satisfy $3\alpha_i+2\alpha_j\in\Delta_+^V$, $5\alpha_i+3\alpha_j\notin\Delta_+^V$, and
$q_{ii}^3\widetilde{q_{ij}}, q_{ii}^4\widetilde{q_{ij}}\neq 1$,
\begin{equation}\label{eqn:relacion con 2alpha i+alpha j, caso 2}
[x_{iij}, x_{3\alpha_i+2\alpha_j}]_c;
\end{equation}

\noindent $\diamond$ if $i,j\in\{1,\ldots,\theta\}$ satisfy $4\alpha_i+3\alpha_j\in\Delta_+^V$, $5\alpha_i+4\alpha_j\notin\Delta_+^V$,
\begin{equation}\label{eqn:relacion (m+1)alpha i+m alpha j, caso 3}
x_{5\alpha_i+4\alpha_j}=[x_{4\alpha_i+3\alpha_j}, x_{ij} ]_c;
\end{equation}

\noindent $\diamond$ if $i,j\in\{1,\ldots,\theta\}$ satisfy $5\alpha_i+2\alpha_j\in\Delta_+^V$, $7\alpha_i+3\alpha_j \notin \Delta_+^V$,
\begin{equation}\label{eqn:relacion con 2alpha i+alpha j, caso 1}
[[x_{iiij}, x_{iij}], x_{iij} ]_c;
\end{equation}

\noindent $\diamond$ if $i,j\in\{1,\ldots,\theta\}$ satisfy $q_{jj}=-1$, $5\alpha_i+4\alpha_j\in\Delta_+^V$,
\begin{equation}\label{eqn:relacion potencia alta}
[x_{iij},x_{4\alpha_i+3\alpha_j}]_c- \frac{b-(1+q_{ii})(1-q_{ii}\zeta)(1+\zeta+q_{ii}\zeta^2)q_{ii}^6\zeta^4} {a\ q_{ii}^3q_{ij}^2q_{ji}^3} x_{3\alpha_i+2\alpha_j}^2,
\end{equation}
where $\zeta=\widetilde{q_{ij}}$, $a=(1-\zeta)(1-q_{ii}^4\zeta^3)-(1-q_{ii}\zeta)(1+q_{ii})q_{ii}\zeta$, $b=(1-\zeta)(1-q_{ii}^6\zeta^5)-a\ q_{ii}\zeta$.\qed
\end{thm}

We fix a realization of $(V,c)$ as a Yetter-Drinfeld module over an abelian group $\Gamma$, i.e., there exist $g_i\in\Gamma$,
$\chi_i\in\widehat{\Gamma}$ such that $\chi_j(g_i)=q_{ij}$ and we make $V$ an object of $\ydg$ by declaring  $x_i\in V_{g_i}^{\chi_i}$.
Let $\cR_V$ be the set of relations defining $\cB(V)$ according to the previous theorem. Note that $\kk\cR_V$ is a Yetter-Drinfeld
submodule of $T(V)$, because each relation is $\Z^\theta$-homogeneous. For each $R\in\cR_V$ of degree $(a_1,\ldots,a_\theta)\in\Z^\theta$, set
\begin{equation}
g_R:= g_1^{a_1}\cdots g_\theta^{a_\theta} , \quad \chi_R:= \chi_1^{a_1}\cdots \chi_\theta^{a_\theta}, \qquad \text{ so }R\in T(V)_{g_R}^{\chi_R}.
\end{equation}
The \emph{support} of $R\in\cR_V$ is the set $\supp R:=\{i\;|\; a_i\neq 0\}$, i.e., the set of indices of letters $x_i$ appearing in $R$.

\begin{prop}\label{prop:gchi different}
For every $R\in\cR_V$ and $t\in\{1,2,\ldots,\theta\}$, we have $(g_R,\chi_R)\neq (g_t,\chi_t)$.
\end{prop}

\begin{proof}
We prove this for each defining relation. For \eqref{eqn:relacion quantum Serre}, see \cite[Proposition 3.1]{AnGa}; the proof does not use that the braiding is of standard type.

We discard easily the cases \eqref{eqn:potencia raices}, \eqref{eqn:potencia raices simples}, \eqref{eqn:relacion dos vertices con -1}, \eqref{eqn:relacion super C3}\vv, \eqref{eqn:relacion super F4-2}\vii, \eqref{eqn:relacion especial rank3}, \eqref{eqn:relacion potencia alta} because $\chi_R(g_R)=1$.

For the remaining cases, note that the propositions in \cite[Section 3]{Ang-crelle} show that $(g_R,\chi_R)\neq (g_t,\chi_t)$ for each $t\notin \supp R$. 
Therefore, we have to consider only the case $t\in\supp(R)$.

For each remaining relation $R$, we compute $\chi_R(g_R)$ and/or $\{ \chi_R(g_t)\chi_t(g_R) \;|\; t\in\supp R\}$.
\bigskip

\eqref{eqn:relacion vertice -1}: we have $\chi_R(g_R)=q_{ii}q_{kk}\neq q_{ii},q_{kk}$. Suppose that $g_R=g_j$, $\chi_R=\chi_j$. Then $\widetilde{q_{ij}}=\chi_R(g_i)\chi_i(g_R)=(q_{ii}\widetilde{q_{ij}})^2$ and $\widetilde{q_{kj}}=\chi_R(g_k)\chi_k(g_R)=(q_{kk}\widetilde{q_{kj}})^2$, so $q_{ii}^2\widetilde{q_{ij}}=q_{kk}^2\widetilde{q_{kj}}=1$. But such a generalized Dynkin diagram is not in \cite{H-classif}, a contradiction.
\smallskip

\eqref{eqn:relacion estandar B2}: now $\chi_R(g_R)=q_{ii}^3\neq q_{ii}$, $\widetilde{q_{ij}}\neq\chi_R(g_i)\chi_i(g_R)=\widetilde{q_{ij}}^2$, so
$(g_R,\chi_R)\neq (g_i,\chi_i), (g_j,\chi_j)$.
\smallskip

\eqref{eqn:relacion estandar B3}: for both sets of conditions, $\widetilde{q_{ij}}q_{jj}^2\widetilde{q_{jk}}=1$ so $\chi_R(g_R)=q_{ii}q_{kk}\neq q_{ii},q_{kk}$. Suppose that $g_R=g_j$, $\chi_R=\chi_j$. But $\widetilde{q_{ij}}\neq\chi_R(g_i)\chi_i(g_R)=\widetilde{q_{ij}}^2$, a contradiction.

\eqref{eqn:relacion triangulo}: recall that $\widetilde{q_{ij}} \widetilde{q_{ik}} \widetilde{q_{jk}} =1$. Suppose that $g_R=g_i$, $\chi_R=\chi_i$.  Then $q_{ii}=\chi_R(g_i)=q_{ii}q_{ij}q_{ik}$, so $q_{ij}q_{ik}=1$. Also $q_{ji}q_{ki}=1$, so $\widetilde{q_{ij}}\widetilde{q_{ik}}=1$ and then $\widetilde{q_{jk}}=1$, a contradiction.
\smallskip

\eqref{eqn:relacion super C3}\vi: simply note that $\chi_R(g_R)=-q_{kk}\neq -1,q_{kk}$.
\smallskip

\eqref{eqn:relacion super C3}\vii: as $\chi_R(g_R)=q_{ii}q_{kk}\neq q_{ii}, q_{kk}$, the remaining case is $t=j$. But also $\widetilde{q_{ij}}=-1\neq \chi_R(g_i)\chi_i(g_R)=-q_{ii}$.
\smallskip

\eqref{eqn:relacion super C3}\viii: it follows since $\chi_R(g_R)=-q_{ii}\neq -1,q_{ii}$.
\smallskip

\eqref{eqn:relacion super C3}\viv again $\chi_R(g_R)=q_{ii}q_{kk}\neq q_{ii}, q_{kk}$, so the remaining case is $t=j$. Suppose that $g_R=g_j$, $\chi_R=\chi_j$,
so $1=q_{jj}^2=\chi_R(g_j)\chi_j(g_R)=\widetilde{q_{ij}}^2\widetilde{q_{jk}}=-q_{ii}$, a contradiction.
\smallskip

\eqref{eqn:relacion super G3}: it follows since  $\chi_R(g_R)=-q_{kk}\neq -1,q_{kk}$.
\smallskip

\eqref{eqn:relacion super C3 raiz de orden 3}: again $\chi_R(g_R)=q_{ii}q_{kk}\neq q_{ii},q_{kk}$. Suppose that $g_R=g_j$, $\chi_R=\chi_j$, so
\begin{align*}
q_{jj}&=q_{ii}q_{kk}, & 1 &=\widetilde{q_{ij}}\widetilde{q_{jk}}=\chi_R(g_i)\chi_i(g_R)\chi_R(g_k)\chi_k(g_R)=q_{ii}^2q_{kk}^2=q_{jj}^2,
\end{align*}
which is a contradiction.
\smallskip

\eqref{eqn:relacion fila 18 rango 3}: it follows from $\chi_R(g_R)=q_{jj}^{-2}\neq q_{ii},q_{jj},q_{kk}$.
\smallskip

\eqref{eqn:relacion fila 18 rango 3, caso 2}: it follows from $\chi_R(g_R)=q_{jj}\neq q_{ii},q_{kk}$, and $\chi_R(g_i)\chi_i(g_R)=1\neq \widetilde{q_{ij}}$.
\smallskip

\eqref{eqn:relacion super G3 raiz de orden 4}: the proof is analogous to the one for \eqref{eqn:relacion super C3 raiz de orden 3}.
\smallskip

\eqref{eqn:relacion parecida a super C3}: As $\chi_R(g_R)=q_{jj}^2q_{kk}$ and $q_{jj}\neq\pm1$, we discard the case $t=k$. The case $t=j$ is also discarded because $1=\chi_R(g_i)\chi_i(g_R)\neq \widetilde{q_{ij}}$. Finally suppose that $\chi_R=\chi_i$, $g_R=g_i$, so $-1=\widetilde{q_{ij}} =\chi_R(g_j)\chi_j(g_R)=q_{jj}^3$. Then $q_{jj}\in\G_6'$ and $-1=\chi_R(g_R)=q_{jj}^2q_{kk}$, so $q_{kk}=q_{jj}$. But this case corresponds to a diagram which is not in \cite{H-classif}, a contradiction.
\smallskip

\eqref{eqn:relacion parecida a super C3-bis}: Note that $\chi_R(g_R)=q_{jj}^2\neq q_{jj},-1=q_{ii}=q_{kk}$ because $q_{jj}^2=\widetilde{q_{ij}}^2\in\G_3'$.
\smallskip

\eqref{eqn:relacion super C4}: simply $\chi_R(g_R)=-q_{ii}\neq q_{ii},q_{jj},q_{kk},q_{ll}$ in all the possible cases.
\smallskip

\eqref{eqn:relacion super C4 modificada}: for $t=l$ we have that $\chi_R(g_R)=q_{jj}^3q_{ll}\neq q_{ll}$, and for $t=i,k$ we have $\chi_R(g_j)\chi_j(g_R) =1 \neq \widetilde{q_{ij}}, \widetilde{q_{kj}}$. Suppose that $\chi_R=\chi_j$ and $g_R=g_j$. Then $\widetilde{q_{ij}}=\chi_R(g_i)\chi_i(g_R)= \widetilde{q_{ij}}^3$, which is a contradiction because $\widetilde{q_{ij}}\neq \pm1$.
\smallskip

\eqref{eqn:relacion super F4-1}: now, $\chi_R(g_i)\chi_i(g_R)=\chi_R(g_j)\chi_j(g_R)=1\neq \widetilde{q_{ij}}, \widetilde{q_{kj}}$, so we discard the cases $t=i,j,k$. Now $\widetilde{q_{kl}}=q_{kk}^{-1}\neq \chi_R(g_k)\chi_k(g_R)=q_{kk}$ so also $(\chi_R,g_R)\neq (\chi_l,g_l)$.
\smallskip

\eqref{eqn:relacion super F4-2}\vi: again $\chi_R(g_i)\chi_i(g_R)=\chi_R(g_j)\chi_j(g_R)=1\neq \widetilde{q_{ij}}, \widetilde{q_{kj}}$, and the cases $t=i,j,k$ are solved. As $\widetilde{q_{kl}}=q_{jj}^3\neq \chi_R(g_k)\chi_k(g_R)=q_{jj}$, we conclude that $(\chi_R,g_R)\neq (\chi_l,g_l)$.
\smallskip

\eqref{eqn:relacion mji,mjk=2}: for $t=j,k$ note that $\chi_R(g_i)\chi_i(g_R)=\widetilde{q_{ij}}\widetilde{q_{ik}}\neq \widetilde{q_{ij}}, \widetilde{q_{ik}}$. For $(\chi_R,g_R)=(\chi_i,g_i)$,
\begin{align*}
q_{ii}&=\chi_R(g_R)=-q_{jj}q_{kk}, &  \widetilde{q_{ij}}&=\chi_R(g_j)\chi_j(g_R)=\widetilde{q_{ij}}^3 q_{jj}^2, & \widetilde{q_{ik}}&=\chi_R(g_k)\chi_k(g_R)=\widetilde{q_{ik}}^3 q_{kk}^2,
\end{align*}
so $q_{jj}=-q_{kk}=\pm q_{ii}^2$, but this diagram is not in \cite{H-classif}, a contradiction.
\smallskip

\eqref{eqn:relacion mij,mji mayor que 1}: we look for the possible generalized Dynkin diagrams for which we need $R$.
\begin{description}
  \item[$\xymatrix{\circ^{\zeta^4}\ar@{-}[r]^{\zeta^9}&\circ^{\zeta^8}}$, $\zeta\in\G_{12}'$] $\chi_R(g_R)=1\neq q_{ii}, q_{jj}$.
  \item[$\xymatrix{\circ^{\zeta^8}\ar@{-}[r]^{\zeta}&\circ^{\zeta^8}}$, $\zeta\in\G_{12}'$] $\chi_R(g_i)\chi_i(g_R)=\chi_R(g_j)\chi_j(g_R)=\zeta^{10}\neq \widetilde{q_{ij}}$.
  \item[$\xymatrix{\circ^{-\zeta}\ar@{-}[r]^{\zeta^7}&\circ^{\zeta^3}}$, $\zeta\in\G_{9}'$] $\chi_R(g_R)=\zeta^8\neq q_{ii}, q_{jj}$.
  \item[$\xymatrix{\circ^{\zeta^6}\ar@{-}[r]^{\zeta^{11}}&\circ^{\zeta^8}}$, $\zeta\in\G_{24}'$] $\chi_R(g_R)=\zeta^4\neq q_{ii}, q_{jj}$.
  \item[$\xymatrix{\circ^{-\zeta}\ar@{-}[r]^{-\zeta^{12}}&\circ^{\zeta^5}}$, $\zeta\in\G_{15}'$] $\chi_R(g_R)=\zeta^{12}\neq q_{ii}, q_{jj}$.
\end{description}
\smallskip

\eqref{eqn:relacion mij mayor que dos, raiz alta}: we consider each possible generalized Dynkin diagram.
\begin{description}
  \item[$\xymatrix{\circ^{-\zeta}\ar@{-}[r]^{\zeta^3}&\circ^{-1}}$, $\zeta\in\G_{5}'$] $\chi_R(g_R)=1\neq q_{ii}, q_{jj}$.
  \item[$\xymatrix{\circ^{\zeta^3}\ar@{-}[r]^{-\zeta^4}&\circ^{-\zeta^{11}}}$, $\zeta\in\G_{15}'$] $\chi_R(g_R)=\zeta^{11}\neq q_{ii}, q_{jj}$.
  \item[$\xymatrix{\circ^{\zeta^8}\ar@{-}[r]^{\zeta^3}&\circ^{-1}}$, $\zeta\in\G_{20}'$] $\chi_R(g_R)=\zeta^{12}\neq q_{ii}, q_{jj}$.
  \item[$\xymatrix{\circ^{\zeta^8}\ar@{-}[r]^{\zeta^{13}}&\circ^{-1}}$, $\zeta\in\G_{20}'$] $\chi_R(g_R)=\zeta^{12}\neq q_{ii}, q_{jj}$.
  \item[$\xymatrix{\circ^{-\zeta^3}\ar@{-}[r]^{\zeta^3}&\circ^{-1}}$, $\zeta\in\G_{7}'$] $\chi_R(g_R)=\zeta^2\neq q_{ii}, q_{jj}$.
  \item[$\xymatrix{\circ^{\zeta^2}\ar@{-}[r]^{\zeta^3}&\circ^{-1}}$, $\zeta\in\G_{8}'$] $\chi_R(g_R)=1\neq q_{ii}, q_{jj}$.
\end{description}
\smallskip

\eqref{eqn:relacion (m+1)alpha i+m alpha j, caso 2}: again consider each possible generalized Dynkin diagram.
\begin{description}
  \item[$\xymatrix{\circ^{\zeta^4}\ar@{-}[r]^{\zeta^{11}}&\circ^{-1}}$, $\zeta\in\G_{12}'$] $\chi_R(g_R)=\zeta^{10}\neq q_{ii}, q_{jj}$.
  \item[$\xymatrix{\circ^{\zeta^8}\ar@{-}[r]^{\zeta^7}&\circ^{-1}}$, $\zeta\in\G_{12}'$] $\chi_R(g_R)=\zeta^2\neq q_{ii}, q_{jj}$.
  \item[$\xymatrix{\circ^{\zeta^8}\ar@{-}[r]^{\zeta^3}&\circ^{-1}}$, $\zeta\in\G_{24}'$] $\chi_R(g_i)\chi_i(g_R)=\zeta$, $\chi_R(g_j)\chi_j(g_R)=1\neq \widetilde{q_{ij}}$.
  \item[$\xymatrix{\circ^{\zeta^6}\ar@{-}[r]^{\zeta}&\circ^{-1}}$, $\zeta\in\G_{24}'$] $\chi_R(g_R)=\zeta^{15}\neq q_{ii}, q_{jj}$.
  \item[$\xymatrix{\circ^{-\zeta}\ar@{-}[r]^{-\zeta^{12}}&\circ^{\zeta^5}}$, $\zeta\in\G_{15}'$] $\chi_R(g_R)=\zeta^{10}\neq q_{ii}, q_{jj}$.
\end{description}
\smallskip

\eqref{eqn:relacion con 2alpha i+alpha j, caso 2}: the unique diagram is $\xymatrix{\circ^{\zeta^3}\ar@{-}[r]^{\zeta^8}&\circ^{-1}}$, $\zeta\in\G_{9}'$, and $\chi_R(g_R)=-\zeta^6\neq q_{ii},q_{jj}$.
\smallskip

\eqref{eqn:relacion (m+1)alpha i+m alpha j, caso 3}: we consider each possible generalized Dynkin diagram.
\begin{description}
  \item[$\xymatrix{\circ^{\zeta}\ar@{-}[r]^{\zeta^2}&\circ^{-1}}$, $\zeta\in\G_{5}'$] $\chi_R(g_R)=1\neq q_{ii}, q_{jj}$.
  \item[$\xymatrix{\circ^{\zeta}\ar@{-}[r]^{\zeta^{17}}&\circ^{-1}}$, $\zeta\in\G_{20}'$] $\chi_R(g_R)=\zeta^5\neq q_{ii}, q_{jj}$.
  \item[$\xymatrix{\circ^{\zeta^{11}}\ar@{-}[r]^{\zeta^7}&\circ^{-1}}$, $\zeta\in\G_{20}'$] $\chi_R(g_R)=\zeta^{15}\neq q_{ii}, q_{jj}$.
  \item[$\xymatrix{\circ^{\zeta^3}\ar@{-}[r]^{-\zeta^4}&\circ^{-\zeta^{11}}}$, $\zeta\in\G_{15}'$] $\chi_R(g_R)=\zeta\neq q_{ii}, q_{jj}$.
  \item[$\xymatrix{\circ^{\zeta^5}\ar@{-}[r]^{-\zeta^{13}}&\circ^{-1}}$, $\zeta\in\G_{15}'$] $\chi_R(g_R)=\zeta^{10}\neq q_{ii}, q_{jj}$.
\end{description}
\smallskip

\eqref{eqn:relacion con 2alpha i+alpha j, caso 1}: the unique diagram is $\xymatrix{\circ^{\zeta^3}\ar@{-}[r]^{-\zeta^2}&\circ^{-1}}$, $\zeta\in\G_{9}'$, and $\chi_R(g_R)=\zeta^9\neq q_{ii},q_{jj}$.
\end{proof}

\begin{thm}\label{thm:homzero}
Suppose $V$ is an object in $\ydg$ such that its Nichols algebra has a finite root system. Then $\Hom^{\kk\Gamma}_{\kk\Gamma}(\kk\cR_V,V)=0$.
\end{thm}
\begin{proof}
If $f\in \Hom^{\kk\Gamma}_{\kk\Gamma}(\kk\cR_V,V)$ and $R\in\cR_V$, then $f(R)\in V^{\chi_R}_{g_R}$.
By Proposition \ref{prop:gchi different}, $V^{\chi_R}_{g_R}=0$ for each $R\in\cR_V$, so $f=0$.
\end{proof}

\begin{thm}\label{thm:diagonal}
If $\cB(V)$ is a Nichols algebra of diagonal type with finite root system then $\cB(V)$ does not admit nontrivial graded deformations as a braided bialgebra.
\end{thm}

\begin{proof}
We fix a realization of $(V,c)$ in $\ydg$ where $\Gamma$ is an abelian group.
Without loss of generality, we may assume that the $g_i$'s generate $\Gamma$ and the $\chi_i$'s generate $\widehat{\Gamma}$.
By Theorem \ref{thm:homzero} and Remark \ref{split}, the conditions needed to invoke Theorem \ref{thm:suff_cond} are satisfied, so $\cB(V)$ does not admit nontrivial
graded deformations in $\ydg$. But our choice of realization ensures that any graded deformation of $\cB(V)$ is in $\ydg$ and hence must be trivial.
\end{proof}

\section{Examples}\label{s:other}

\subsection{Positive parts of quantum groups}
It is well known that, in the generic case, the positive part of a quantized enveloping algebra is a Nichols algebra of diagonal type.
By Theorem \ref{thm:diagonal}, these positive parts are rigid.
More generally, this applies to the ``diagram'' of the pointed Hopf algebra $U(\cD)$ associated to a generic datum $\cD$ of finite Cartan type --- see \cite{ASqg},
where it is shown that any pointed Hopf algebra whose group-like elements form a finitely generated abelian group is isomorphic to some $U(\cD)$
if it is a domain with finite Gelfand-Kirillov dimension and its infinitesimal braiding is positive.

\subsection{Distinguished pre-Nichols algebras}
These are infinite-dimensional braided Hopf algebras projecting onto the corresponding finite-dimensional Nichols algebras.
They were formally defined in \cite[Definition 3.1]{A-prenichols} generalizing the situation with quantum groups at roots of unity and the corresponding small quantum groups.
Let $V$ be a braided vector space of diagonal type such that $\cB(V)$ is finite-dimensional. Then the distinguished pre-Nichols algebra
$\widetilde{\cB}(V)$ is the quotient of $T(V)$ by the relations in Theorem \ref{thm:presentacion minima} except the powers of root vectors \eqref{eqn:potencia raices}.
As a consequence of Theorem \ref{thm:homzero}, we have:

\begin{thm}\label{thm:diagonal-pre-Nichols}
Let $(V,c)$ be a braided vector space of diagonal type such that $\cB(V)$ is finite-dimensional.
Then $\widetilde{\cB}(V)$ does not admit nontrivial graded deformations as a braided bialgebra.\qed
\end{thm}

\subsection{Nichols algebras over dihedral groups} Let $D_m$ denote the dihedral group of order $2m$.
For odd $m$, it is not known whether the category of Yetter-Drinfeld modules over $D_m$ has any finite-dimension Nichols algebras.
For even $m\ge 4$, the only known finite-dimensional Nichols algebras have a symmetric braiding \cite{FG}, so Theorem \ref{thm:symmetric} applies.

\subsection{Nichols algebras over symmetric groups}

Let $n\geq3$. The quadratic algebra $\mathcal{FK}_{n}$, introduced by Fomin and Kirillov \cite{FK}, is presented by generators $x_{(ij)}$, $1\leq i<j\leq n$,
and relations
\begin{align*}
x_{(ij)}^{2}&=0,&  &1\leq i<j\leq n,\\
x_{(ij)}x_{(jk)}&=x_{(jk)}x_{(ik)}+x_{(ik)}x_{(ij)},&  &1\leq i<j<k\leq n,\\
x_{(jk)}x_{(ij)}&=x_{(ik)}x_{(jk)}+x_{(ij)}x_{(ik)},&  &1\leq i<j<k\leq n,\\
x_{(ij)}x_{(kl)}&=x_{(kl)}x_{(ij)},&  & \# \{i,j,k,l\}=4.
\end{align*}
Milinski and Schneider \cite{MiS} showed how to make $\mathcal{FK}_n$ a graded bialgebra in the category of Yetter-Drinfeld modules over the symmetric group $S_n$.
As an algebra, it is generated by the vector space $V_n$ with basis $\{x_{(ij)}\mid 1\leq i<j\leq n\}$. Identifying $(ij)$ with the corresponding transposition in $S_n$,
we can make $V_n$ a Yetter-Drinfeld module where the coaction is defined by declaring $x_\sigma$ a homogeneous element of degree $\sigma$, and the action is the
conjugation twisted by the sign. The corresponding braiding on $V_n$ is given by
\[
c(x_\sigma\otimes x_\tau)=\chi(\sigma,\tau)x_{\sigma\tau\sigma^{-1}}\otimes x_\sigma,\quad\quad
 \chi(\sigma,\tau)=
 \begin{cases}
 	1 & \sigma(i)<\sigma(j), \tau=(ij), \, i<j,\\
  -1 & \text{otherwise,}
\end{cases}
\]
where $\sigma$ and $\tau$ are transpositions. Then the above relations generate a biideal in the (braided) tensor bialgebra $T(V_n)$.

It is easy to see that $\mathcal{FK}_n$ projects onto the Nichols algebra $\cB(V_n)$. For $n=3,4,5$, it is known that $\mathcal{FK}_n=\cB(V_n)$ and
has dimension, respectively, $12$, $576$ and $8294400$ (see \cite{MiS} for $n=3,4$ and \cite{Gr} for $n=5$). Milinski and Schneider conjectured that $\mathcal{FK}_n$
coincides with $\cB(V_n)$ for all $n$. Moreover, it has been conjectured that $\dim\mathcal{FK}_n=\infty$ for $n\geq6$ \cite{FK}.

\begin{thm}\label{thm:Fomin-Kirillov}
Let $n\geq 3$. Then $\mathcal{FK}_n$ does not admit nontrivial graded deformations as a braided bialgebra.
\end{thm}
\begin{proof}
All relations are in degree $2$ and cannot have coaction given by transposition.  As the only primitives in degrees smaller than $2$ are in degree $1$ and have coaction given by transpositions, the assumption of Theorem \ref{thm:suff_cond} is satisfied and these algebras are rigid.
\end{proof}

\end{document}